\documentclass[10pt]{article}

\usepackage[hmargin=35mm,vmargin=40mm]{geometry}

\usepackage{amsmath}      
\usepackage{amssymb}      
\usepackage{amsthm}       
\usepackage{enumerate}    
\usepackage{graphicx}     
\usepackage{mathrsfs}     
\usepackage{url}          

\newtheorem{theorem}{Theorem}[section]
\newtheorem{corollary}[theorem]{Corollary}
\newtheorem{lemma}[theorem]{Lemma}

\newtheorem*{notation}{Notation}

\theoremstyle{definition}
\newtheorem{definition}[theorem]{Definition}
\newtheorem*{assumptions}{Assumptions}
\newtheorem*{remarks}{Remarks}

\numberwithin{equation}{section}

\newcommand{\cone}[1]{#1^\vee}
\newcommand{\conv}{\mathrm{conv}}
\newcommand{\fpi}{\pi}
\newcommand{\latt}{L}
\newcommand{\lattadm}{\latt_A}
\newcommand{\N}{\mathbb{N}}

\newcommand{\qproj}{\mathscr{Q}}
\newcommand{\qcone}{\cone{\qproj}}
\newcommand{\R}{\mathbb{R}}

\newcommand{\sproj}{\mathscr{S}}
\newcommand{\scone}{\cone{\sproj}}
\newcommand{\tet}[1]{\tau^{(#1)}}
\newcommand{\tri}{\mathcal{T}}
\newcommand{\Z}{\mathbb{Z}}

\newcommand{\co}{\colon\thinspace}

\title{Maximal admissible faces and asymptotic bounds \\
    for the normal surface solution space}
\author{Benjamin A.~Burton}

\date{December 10, 2010}

\begin{document}

\maketitle

\begin{abstract}
    The enumeration of normal surfaces is a key bottleneck in
    computational three-di\-men\-sion\-al topology.  The underlying procedure
    is the enumeration of admissible vertices of a high-dimensional
    polytope, where admissibility is a powerful but non-linear and
    non-convex constraint.
    The main results of this paper are significant improvements upon the
    best known asymptotic bounds on the number of admissible vertices, using
    polytopes in both the standard normal surface coordinate system and the
    streamlined quadrilateral coordinate system.

    To achieve these results we examine the layout of admissible points
    within these polytopes.  We show that these points correspond to
    well-behaved substructures of the face lattice, and we study
    properties of the corresponding ``admissible faces''.
    Key lemmata include upper bounds on the number of
    maximal admissible faces of each dimension, and a bijection between
    the maximal admissible faces in the two coordinate systems mentioned above.


\end{abstract}

%
%

\section{Introduction} \label{s-intro}

Computational topology in three dimensions is a diverse and expanding
field, with algorithms drawing on a range of ideas from
geometry, combinatorics, algebra, analysis, and operations research.
A key tool in this field is \emph{normal surface theory}, which
allows us to convert difficult topological decision and decomposition
problems into more tractable
enumeration and optimisation problems over convex polytopes and polyhedra.

In this paper we develop new asymptotic bounds on the complexity
of problems in normal surface theory, which in turn impacts upon a wide
range of topological algorithms.  The techniques that we use are based on
ideas from polytope theory, and the bulk of this paper focuses on the
combinatorics of the various polytopes and polyhedra that arise in
the study of normal surfaces.

Normal surface theory was introduced by Kneser \cite{kneser29-normal},
and further developed by Haken
\cite{haken61-knot,haken62-homeomorphism}
and Jaco and Oertel \cite{jaco84-haken} for use in algorithms.
The core machinery of normal surface theory is now central to many important
algorithms in three-dimensional topology, including
unknot recognition \cite{haken61-knot},
3-sphere recognition
\cite{burton10-quadoct,jaco03-0-efficiency,rubinstein95-3sphere,
thompson94-thinposition},
connected sum decomposition
\cite{jaco03-0-efficiency,jaco95-algorithms-decomposition},
and testing for embedded incompressible surfaces
\cite{burton09-ws,jaco84-haken}.

The core ideas behind normal surface theory are as follows.
Suppose we are searching for an ``interesting'' surface embedded within a
3-manifold (such as a disc bounded by the unknot, or a sphere that
splits apart a connected sum).  We construct a high-dimensional
convex polytope called the \emph{projective solution space}, and we
define the \emph{admissible} points within this polytope to be
those that satisfy an additional set of non-linear and non-convex
constraints.  The importance of this polytope is that every admissible
and rational point within it corresponds to an embedded surface within our
3-manifold, and moreover all embedded ``normal'' surfaces within our
3-manifold are represented in this way.

We then prove that, if any interesting surfaces exist, at least one
must be represented by a \emph{vertex}
of the projective solution space.  Our algorithm is now straightforward:
we construct this polytope, enumerate its admissible vertices,
reconstruct the corresponding surfaces, and test whether any of
these surfaces is ``interesting''.

The development of this machinery was a breakthrough in
computational topology.  However, the algorithms that it produces
are often extremely slow.  The main bottleneck lies in enumerating the
admissible vertices of the projective solution space---polytope vertex
enumeration is NP-hard in general \cite{dyer83-complexity,khachiyan08-hard},
and there is no evidence to suggest that our particular polytope is
simple enough or special enough to circumvent this.\footnote{%
    In fact, Agol et~al.\ have proven that the \emph{knot genus}
    problem is NP-complete \cite{agol02-knotgenus}.
    The knot genus algorithm uses normal surface theory, but in
    a more complex way than we describe here.}

Nevertheless, there \emph{is} strong evidence to suggest that these
procedures can be made significantly faster than current theoretical
bounds imply.  For instance, detailed experimentation with the
quadrilateral-to-standard conversion procedure---a key step in the
current state-of-the-art enumeration algorithm---suggests that this
conversion runs in small polynomial time, even though the best
theoretical bound remains exponential \cite{burton09-convert}.
Comprehensive experimentation with the projective solution space
\cite{burton10-complexity} suggests that
the number of admissible vertices, though exponential, grows at a rate
below $O(1.62^n)$ in the average case and around $O(2.03^n)$ in the
worst case, compared to the best theoretical bound of approximately
$O(29.03^n)$ (which we improve upon in this paper).  Here
the ``input size'' $n$ is the number of tetrahedra in the underlying
3-manifold triangulation.

The key to this improved performance is our \emph{admissibility}
constraint.  Admissibility is a powerful constraint that eliminates
almost all of the complexity of the projective solution space (we see
this vividly illustrated in Section~\ref{s-lattice}).  However, as a non-linear
and non-convex constraint it is difficult to weave admissibility
into complexity arguments, particularly if we wish to draw on the
significant body of work from the theory of convex polytopes.

The ultimate aim of this paper is to bound the number of admissible vertices
of the projective solution space.  This is a critical quantity for the
running times of normal surface algorithms.  First, however well we
exploit admissibility in our vertex enumeration algorithms, running times
\emph{must} be at least as large as the output size---that is, the
number of admissible vertices.  Moreover, for some topological algorithms,
the procedure that we perform on each admissible vertex is significantly
slower than the enumeration of these vertices (see Hakenness testing
for an example \cite{burton09-ws}).  In these cases, the
number of admissible vertices becomes a central factor in the overall
running time.

Enumeration algorithms typically work in one of two coordinate systems:
\emph{standard coordinates} of dimension $7n$, and \emph{quadrilateral
coordinates} of dimension $3n$.
The strongest bounds known to date are as follows:
\begin{itemize}
    \item In standard coordinates, the first bound on the number of
    admissible vertices of the projective solution space was
    $128^n$, due to Hass et~al.\ \cite{hass99-knotnp}.
    The author has recently refined this bound to
    $O(\phi^{7n}) \simeq O(29.03^n)$, where $\phi$ is the golden ratio
    \cite{burton10-complexity}.\footnote{The paper
        \cite{burton10-complexity} also places
        a \emph{lower} bound on the worst case complexity of
        $\Omega(17^{n/4}) \simeq \Omega(2.03^n)$.}

    \item In quadrilateral coordinates, the best general bound is
    $4^n$ (this bound does not appear in the literature but is
    well known, and we outline the simple proof in
    Section~\ref{s-prelim-tri}).

    \item In the case where the input is a one-vertex triangulation,
    the author sketches a bound of approximately $O(15^n/\sqrt{n})$
    admissible vertices in standard coordinates \cite{burton10-dd}.
    This case is important
    for practical computation, as we discuss further in
    Section~\ref{s-prelim}.
\end{itemize}

The main results of this paper are as follows.
In standard coordinates, we tighten the general bound from approximately
$O(29.03^n)$ to $O(14.556^n)$ (Theorem~\ref{t-std-bound}).
In quadrilateral coordinates, we tighten the general bound from
$4^n$ to approximately $O(3.303^n)$ (Theorem~\ref{t-quad-bound}).
For the one-vertex case in standard coordinates, we strengthen
$O(15^n/\sqrt{n})$ to approximately $O(4.852^n)$
(Theorem~\ref{t-std-bound-v1}).

We achieve these results by studying not just the admissible vertices,
but the broader region formed by \emph{all} admissible points within the
projective solution space.
Although this region is not convex, we show that
it corresponds to a well-behaved structure within the face lattice of the
surrounding
polytope.  By working through maximal elements of this structure---that
is, \emph{maximal admissible faces} of the polytope---we are able to
draw on strong results from polytope theory such as McMullen's
upper bound theorem \cite{mcmullen70-ubt},
yet still enjoy the significant reduction in
complexity that admissibility provides.

To contrast this paper from earlier work:
The bound of $O(29.03^n)$ in \cite{burton10-complexity}
is a straightforward consequence of McMullen's theorem,
applied once to the entire projective solution space without using
admissibility at all.
In this paper, the key innovations are the decomposition of the admissible
region into maximal admissible faces, and the combinatorial analysis
of these maximal admissible faces.  These new techniques allow us to
apply McMullen's theorem repeatedly in a careful and targeted fashion,
ultimately yielding the stronger bounds outlined above.

Throughout this paper, we restrict our attention to closed
and connected 3-manifolds.
In addition to the main results listed above, we also prove several
key lemmata that may be useful in future work.
These include an upper bound of $3^{n-1-d}$ maximal admissible faces
of dimension $d$ in quadrilateral coordinates (Lemma~\ref{l-quad-bound-dim}),
a bijection between maximal admissible faces in quadrilateral
coordinates and standard coordinates (Lemma~\ref{l-bijection-std-quad}),
and a tight upper bound of $n+1$ vertices for any triangulation with
$n > 2$ tetrahedra (Lemma~\ref{l-min-vert}).

The layout of this paper is as follows.
Section~\ref{s-prelim} begins with an overview of relevant
results from normal surface theory and polytope theory.
In Section~\ref{s-lattice} we study the structure of admissible points
in detail, focusing in particular on admissible faces and
maximal admissible faces of the projective solution space.

We turn our attention to asymptotic bounds in Section~\ref{s-ubt},
focusing on properties of the bounds obtained by McMullen's theorem.
In Section~\ref{s-quad} we prove our main results in quadrilateral
coordinates, and in Section~\ref{s-std} we transport these results
to standard coordinates with the help of the aforementioned bijection.
Section~\ref{s-discussion} finishes with a discussion of our techniques,
including experimental comparisons and possibilities for further
improvement.

\section{Preliminaries} \label{s-prelim}

In this section we recount key definitions and results from the two
core areas of normal surface theory and polytope theory.
Section~\ref{s-prelim-tri} covers 3-manifold triangulations and normal
surfaces, and Section~\ref{s-prelim-polytopes} discusses
convex polytopes and polyhedra.

In this brief summary we only give the details
necessary for this paper.  For a more thorough overview of these topics,
the reader is referred to Hass et~al.\ \cite{hass99-knotnp} for the
theory of normal surfaces and its role in computational topology, and to
Gr\"unbaum \cite{grunbaum03} or Ziegler \cite{ziegler95} for the theory
of convex polytopes.

\begin{assumptions}
    The following assumptions and conventions run throughout this paper:
    \begin{itemize}
        \item We always assume that we are working with a closed 3-manifold
        triangulation $\tri$ constructed from precisely $n$ tetrahedra
        (see Section~\ref{s-prelim-tri} for details), and we always
        assume that this triangulation is connected;
        \item The words ``polytope'' and ``polyhedron'' refer
        exclusively to \emph{convex} polytopes and polyhedra;
        \item For convenience, we allow arbitrary integers $a,b$ in the
        binomial coefficients $\binom{a}{b}$ but we define $\binom{a}{b} = 0$
        unless $0 \leq b \leq a$.
    \end{itemize}
\end{assumptions}

\subsection{Triangulations and normal surfaces} \label{s-prelim-tri}

A \emph{closed 3-manifold} is a compact topological space that locally
``looks'' like $\R^3$ at every point.\footnote{More precisely,
a closed 3-manifold is a compact and separable metric space in which every
point has an open neighbourhood homeomorphic to $\R^3$ \cite{hempel76}.}
A \emph{closed 3-manifold triangulation} is a collection of $n$
tetrahedra whose 2-dimensional faces are affinely identified (or
``glued together'') in pairs so that the resulting topological space
is a closed 3-manifold.

We do not require these tetrahedra to be rigidly embedded in some larger
space---in other words, tetrahedra can be ``bent'' or ``stretched''.
In particular, we allow identifications between two faces of the same
tetrahedron; likewise, we may find that multiple edges or vertices of
the same tetrahedron become identified together as a result of our
face gluings.  Some authors refer to such triangulations as
\emph{semi-simplicial triangulations} or \emph{pseudo-triangulations}.
This more flexible definition allows us to represent complex
topological spaces using relatively few tetrahedra, which is extremely
useful for computation.

\begin{figure}[htb]
    \centering
    \includegraphics[scale=0.5]{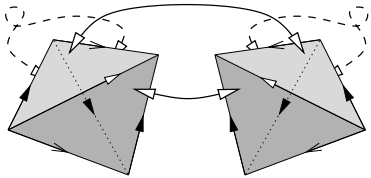}
    \caption{An example of a closed 3-manifold triangulation}
    \label{fig-s2xs1}
\end{figure}

Tetrahedron vertices that become identified together are collectively referred
to as a single \emph{vertex of the triangulation}; similarly for edges and
2-dimensional faces.
Figure~\ref{fig-s2xs1} illustrates a triangulation formed from $n=2$
tetrahedra: the two front faces of the left tetrahedron are identified
directly with the two front faces of the right tetrahedron, and
in each tetrahedron the two back faces are identified together with a
twist.\footnote{%
    The underlying 3-manifold described by this triangulation is the
    product space $S^2 \times S^1$.}
This triangulation has only one vertex (since all eight tetrahedron
vertices become identified together), and it has precisely three edges
(indicated by the three different types of arrowhead).

One-vertex triangulations are of
particular interest to computational topologists, since they often
simplify to very few tetrahedra, and since
some algorithms become significantly simpler and/or faster
in a one-vertex setting.  Several authors have shown that one-vertex
triangulations exist for a wide range of 3-manifolds with a
variety of procedures to construct them; see
\cite{jaco03-0-efficiency,martelli02-decomp,matveev90-complexity}
for details.
We devote particular attention to one-vertex triangulations in
Theorem~\ref{t-std-bound-v1} of this paper.

As indicated earlier, for the remainder of this paper we assume that we
are working with a closed (and connected) 3-manifold triangulation $\tri$
constructed from $n$ tetrahedra.
A \emph{normal surface} within $\tri$ is a closed 2-dimensional surface
embedded within $\tri$ that intersects each tetrahedron of
$\tri$ in a collection of zero or more \emph{normal discs}.
A normal disc is either an embedded \emph{triangle} (meeting three
distinct edges of the tetrahedron) or an embedded \emph{quadrilateral}
(meeting four distinct edges), as illustrated in Figure~\ref{fig-normaldiscs}.

\begin{figure}[htb]
    \centering
    \includegraphics[scale=0.5]{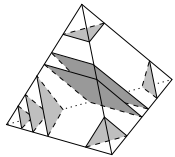}
    \caption{Normal triangles and quadrilaterals within a tetrahedron}
    \label{fig-normaldiscs}
\end{figure}

Like the tetrahedra themselves, triangles and quadrilaterals need not be
rigidly embedded (i.e., they can be ``bent'').  However, they must
intersect the edges of the tetrahedron transversely, and they cannot meet
the vertices of the tetrahedron at all.
Figure~\ref{fig-s2xs1-normal} illustrates a normal surface within the
example triangulation given earlier.\footnote{%
    This surface is an embedded essential 2-dimensional sphere.}
Normal surfaces may be disconnected or empty.

\begin{figure}[htb]
    \centering
    \includegraphics[scale=0.5]{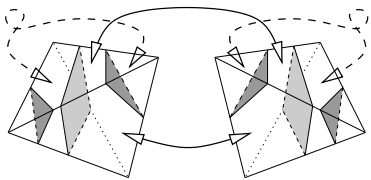}
    \caption{A normal surface within a closed 3-manifold triangulation}
    \label{fig-s2xs1-normal}
\end{figure}

Within each tetrahedron there are four \emph{types} of triangle and
three \emph{types} of quadrilateral, defined by which edges of the
tetrahedron they intersect (Figure~\ref{fig-normaldiscs} includes discs
of all four triangle types but only one of the three quadrilateral types).
We can represent a normal surface by the integer vector
\[ \left(\ t_{1,1},t_{1,2},t_{1,3},t_{1,4},\ q_{1,1},q_{1,2},q_{1,3}\ ;
         \ t_{2,1},t_{2,2},t_{2,3},t_{2,4},\ q_{2,1},q_{2,2},q_{2,3}\ ;
         \ \ldots,q_{n,3}\ \right) \in \Z^{7n}, \]
where each $t_{i,j}$ or $q_{i,j}$ is the number of triangles or
quadrilaterals respectively of the $j$th type within the $i$th tetrahedron.

A key theorem of Haken \cite{haken61-knot} states that an arbitrary
integer vector in $\R^{7n}$ represents a normal surface if and only if:
\begin{enumerate}[(i)]
    \item all coordinates of the vector are non-negative;
    \item the vector satisfies the \emph{standard matching equations},
    which are $6n$ linear homogeneous equations in $\R^{7n}$
    that depend on $\tri$;
    \item the vector satisfies the \emph{quadrilateral constraints},
    which state that for each $i$, at most one of the three
    quadrilateral coordinates $q_{i,1},q_{i,2},q_{i,3}$ is non-zero.
\end{enumerate}
Any vector in $\R^{7n}$ that satisfies all three of these constraints is
called \emph{admissible} (note that we extend this definition to apply
to non-integer vectors).
The quadrilateral constraints are the most problematic of these
three conditions, since they are non-linear constraints with a
non-convex solution set.

We refer to the region of $\R^{7n}$ that satisfies the non-negativity
constraints and the standard matching equations as the
\emph{standard solution cone},
which we denote $\scone$; this is a pointed polyhedral cone in $\R^{7n}$
with apex at the origin.  We also consider the cross-section of this
cone with the \emph{projective hyperplane} $\sum t_{i,j} + \sum q_{i,j} = 1$,
which we call the \emph{standard projective solution space} and denote
$\sproj$; this is a bounded
polytope in $\R^{7n}$.  The admissible vertices of the standard projective
solution space---that is, the vertices that also satisfy the quadrilateral
constraints---are called the \emph{standard solution set}.

Tollefson \cite{tollefson98-quadspace} defines a smaller
vector representation in $\R^{3n}$,
obtained by considering only the quadrilateral coordinates $q_{i,j}$ and
ignoring the triangular coordinates $t_{i,j}$.
This smaller coordinate system is more efficient for computation, but
its use is restricted to a smaller range of topological algorithms.
Tollefson proves a theorem similar to Haken's, in that an arbitrary integer
vector in $\R^{3n}$ represents a normal surface if and only if:
\begin{enumerate}[(i)]
    \item all coordinates of the vector are non-negative;
    \item the vector satisfies the \emph{quadrilateral matching equations},
    which is a smaller family of linear homogeneous equations in $\R^{3n}$
    that again depend on $\tri$;
    \item the vector satisfies the \emph{quadrilateral constraints}
    as defined above.
\end{enumerate}

Again, any vector in $\R^{3n}$ that satisfies all three of these constraints
is called \emph{admissible}.
The region of $\R^{3n}$ that satisfies the non-negativity
constraints and the quadrilateral matching equations is the
\emph{quadrilateral solution cone}, denoted $\qcone$, which is
a pointed polyhedral cone in $\R^{3n}$ with apex at the origin.
The cross-section with the \emph{projective hyperplane} $\sum q_{i,j} = 1$
is likewise called the \emph{quadrilateral projective solution space}
and denoted $\qproj$; this is a bounded polytope in $\R^{3n}$.
The admissible vertices of the quadrilateral projective solution space
are called the \emph{quadrilateral solution set}.

In general, when we work in $\R^{7n}$ we say we are working in
\emph{standard coordinates}, and when we work in $\R^{3n}$ we say we are
working in \emph{quadrilateral coordinates}.  See
\cite{burton09-convert} for a detailed discussion
of the relationship between these coordinate systems as well as
fast algorithms for converting between them.

Enumerating the standard and quadrilateral solution sets is a common
feature of high-level algorithms in 3-manifold topology.
Moreover, this enumeration is often the computational bottleneck, and so
it is important to have fast enumeration algorithms as well as
good complexity bounds on the size of each solution set.  The latter
problem is the main focus of this paper.

As noted in the introduction, the only upper bound to date on the size
of the quadrilateral solution set is the well-known but
unpublished\footnote{%
    Although the bound of $\leq 4^n$ does not appear in the
    literature, an asymptotic bound of
    $O(4^n/\sqrt{n})$ is sketched in \cite{burton10-dd} for the special
    case of a one-vertex triangulation.}
bound of $4^n$.
The proof is simple.
For any vector $\mathbf{x} \in \qproj$, the \emph{zero set} of $\mathbf{x}$
is defined as $\{k\,|\,x_k=0\}$; in other words, the set of indices at which
$\mathbf{x}$ has zero coordinates.  It is shown in \cite{burton10-dd} that
any vertex of $\qproj$ can be completely reconstructed from its zero set.
The quadrilateral constraints allow for at most four different
zero\,/\,non-zero patterns amongst the three quadrilateral coordinates
for each tetrahedron, restricting us to at most $4^n$ distinct zero
sets in total, and therefore at most $4^n$ admissible vertices of $\qproj$.

Two admissible vectors $\mathbf{u},\mathbf{v} \in \R^{7n}$ or
$\mathbf{u},\mathbf{v} \in \R^{3n}$ are said to be
\emph{compatible} if the quadrilateral constraints are satisfied by
both of them together.  That is,
for each $i$, at most one of the three quadrilateral coordinates
$q_{i,1},q_{i,2},q_{i,3}$ can be non-zero in \emph{either}
$\mathbf{u}$ or $\mathbf{v}$.

Some particular vectors in standard and quadrilateral coordinates are
worthy of note:
\begin{itemize}
    \item For each vertex $V$ of the triangulation $\tri$, the
    \emph{vertex link} of $V$ is the vector in $\R^{7n}$ describing a small
    embedded normal sphere surrounding $V$.  This normal surface
    consists of triangles
    only, and so the corresponding vector is zero on all quadrilateral
    coordinates.  If $\tri$ contains $v$ distinct vertices then there
    are $v$ corresponding vertex links, all of which are admissible
    and linearly independent.

    \item For each $i=1,\ldots,n$, the
    \emph{tetrahedral solution} $\tet{i} \in \R^{3n}$
    is the vector with $q_{i,1} = q_{i,2} = q_{i,3} = 1$ and all other
    quadrilateral coordinates equal to zero.
    The tetrahedral solutions were
    introduced by Kang and Rubinstein \cite{kang04-taut1} as part of a
    ``canonical basis'' for normal surface theory.
    They satisfy the quadrilateral matching equations
    (so $\tet{i} \in \qcone$), but they do not satisfy the
    quadrilateral constraints (so $\tet{i}$ is not admissible).
\end{itemize}

There is a natural relationship between standard and quadrilateral
coordinates.  We define the \emph{quadrilateral projection map}
$\pi\co\R^{7n}\to\R^{3n}$ as the map that deletes all $4n$
triangular coordinates $t_{i,j}$ and retains all $3n$ quadrilateral
coordinates $q_{i,j}$.  This map is linear, and it maps the admissible
points of $\scone$ \emph{onto} the admissible points of $\qcone$.
This map is not one-to-one, but the kernel is precisely the subspace of
$\R^{7n}$ generated by the (linearly independent) vertex links.
The relevant results are proven by Tollefson for integer vectors in
\cite{tollefson98-quadspace}; see \cite{burton09-convert} for
extensions into $\R^{7n}$ and $\R^{3n}$.

For points within the solution cones, the quadrilateral projection map
preserves admissibility and inadmissibility, and it preserves
compatibility and incompatibility.  That is,
$\mathbf{v} \in \scone$ is admissible if and only if
$\pi(\mathbf{v}) \in \qcone$ is admissible, and admissible vectors
$\mathbf{u},\mathbf{v} \in \scone$ are compatible if and only if
$\pi(\mathbf{u}),\pi(\mathbf{v}) \in \qcone$ are compatible.

We finish this overview of normal surface theory with an important
dimensional result.  This theorem is due Tillmann \cite{tillmann08-finite},
and extends earlier work of Kang and Rubinstein for non-closed manifolds
\cite{kang04-taut1}.

\begin{theorem}[Tillmann, 2008] \label{t-matching-dim}
    The solution space to the quadrilateral matching equations
    in $\R^{3n}$ has dimension precisely $2n$.
\end{theorem}

\subsection{Polytopes and polyhedra} \label{s-prelim-polytopes}

We follow Ziegler \cite{ziegler95} for our terminology:
\emph{polytopes} are always bounded
(like the projective solution spaces $\sproj$ and $\qproj$),
and \emph{polyhedra} may be bounded or unbounded
(like the solution cones $\scone$ and $\qcone$).
The reader is referred to \cite{ziegler95} for background
material on standard concepts such as faces, facets and supporting
hyperplanes.

In this paper we work with the
\emph{face lattice} of a polytope or polyhedron $P$,
which encodes all of the
combinatorial information about the facial structure of $P$.
Specifically, the face lattice is
the poset consisting of all faces of $P$ ordered by the
subface relation, and is denoted by $\latt(P)$.
See Figure~\ref{fig-lattice-cube} for an illustration in the
case where $P$ is the 3-dimensional cube.

\begin{figure}[htb]
    \centering
    \includegraphics[scale=0.7]{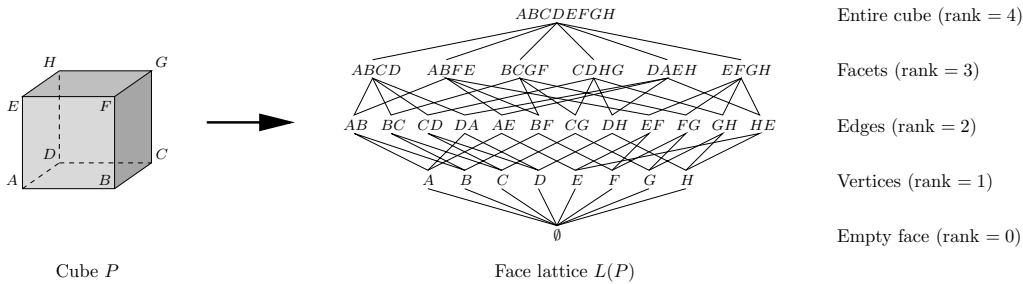}
    \caption{The face lattice of a cube}
    \label{fig-lattice-cube}
\end{figure}

We recount some key properties of the face lattice.
Any two faces $F,G \in \latt(P)$ have a unique
greatest lower bound in $\latt(P)$, called the \emph{meet} $F \wedge G$
(this corresponds to the intersection $F \cap G$), and also a unique
least upper bound in $\latt(P)$, called the \emph{join} $F \vee G$.
There is a unique minimal element of $\latt(P)$ (corresponding to the
empty face) and a unique maximal element of $\latt(P)$
(corresponding to $P$ itself).
Moreover, $\latt(P)$ is a \emph{graded lattice}:
it is equipped with a \emph{rank function}
$r \co \latt(P) \to \N$ defined by $r(F) = \dim F + 1$,
so that whenever $G$ \emph{covers} $F$ in the poset
(that is, $F < G$ and there is no $X$ for which $F < X < G$),
we have $r(G) = r(F) + 1$.
Once again we refer to Ziegler \cite{ziegler95} for
details.

For any polytope $F$, we define the \emph{cone over $F$} to be
$\cone{F} = \{ \lambda \mathbf{x}\,|\,\mathbf{x} \in F,\ \lambda \geq 0\}$.
As a special case, for the empty face $\emptyset$ we define
$\cone{\emptyset} = \{\mathbf{0}\}$.
It is clear that the solution cones $\scone$ and $\qcone$ are indeed
the cones over the projective solution spaces
$\sproj$ and $\qproj$, as the notation suggests.
The facial structures of polytopes and their cones are tightly related, as
described by the following well-known result:

\begin{lemma} \label{l-prelim-cone}
    Let $P$ be a $d$-dimensional polytope whose affine hull does not contain
    the origin.
    Then $\cone{P}$ is a $(d+1)$-dimensional polyhedron, and the cone map
    $F \mapsto \cone{F}$ is a bijection from the faces
    of $P$ to the non-empty faces of $\cone{P}$.
    This bijection maps $i$-faces of $P$ to $(i+1)$-faces of
    $\scone$ for all $i$.  Both the bijection and its inverse
    preserve subfaces; in other words,
    $\cone{F} \subseteq \cone{G}$ if and only if $F \subseteq G$.
\end{lemma}

A celebrated milestone in polytope complexity theory was McMullen's
\emph{upper bound theorem}, proven in 1970 \cite{mcmullen70-ubt}.
In essence, this result places an upper bound on the number
of $i$-faces of a $d$-dimensional $k$-vertex polytope, for any
$i \leq d < k$.
This upper bound is tight, and equality is achieved in
the case of \emph{cyclic polytopes} (and more generally,
\emph{neighbourly simplicial polytopes}).
Taken in dual form, McMullen's theorem bounds the number of $i$-faces
of a $d$-dimensional polytope with $k$ facets.  In this paper
we use the dual form for the case $i=0$, which reduces to the following
result:

\begin{theorem}[McMullen, 1970] \label{t-ubt}
    For any integers $2 \leq d < k$,
    a $d$-dimensional polytope with precisely $k$ facets can have at most
    \begin{equation} \label{eqn-ubt}
    \binom{k - \lfloor \frac{d+1}{2} \rfloor}{k - d} +
    \binom{k - \lfloor \frac{d+2}{2} \rfloor}{k - d}
    \end{equation}
    vertices.\footnote{%
    The expression (\ref{eqn-ubt}) is the number of facets of the
    cyclic $d$-dimensional polytope with $k$ vertices; see a standard
    reference such as Gr\"unbaum \cite{grunbaum03} for details.}
\end{theorem}

\section{Admissibility and the face lattice} \label{s-lattice}

In this section we explore the facial structures of the bounded polytopes
$\sproj$ and $\qproj$ (the standard and quadrilateral projective solution
spaces) and the tightly-related polyhedral cones $\scone$ and $\qcone$
(the standard and quadrilateral solution cones).
In particular we focus on \emph{admissible faces}, which are
faces along which the quadrilateral constraints are always satisfied.

We begin by showing that the admissible faces together contain all
admissible points (that is, all of the ``interesting'' points from the
viewpoint of normal surface theory).  Following this, we study the
layout of admissible faces within the larger face lattice of each
solution space, and we examine the relationships between
admissible faces and pairs of compatible points.
We finish the section by categorising \emph{maximal} admissible faces
in a variety of ways.

\begin{definition}[Admissible face] \label{d-adm-face}
    Let $F$ be a face of the standard projective solution space $\sproj$.
    Then $F$ is an \emph{admissible face} of $\sproj$ if every point in $F$
    satisfies the quadrilateral constraints.  We say that $F$ is a
    \emph{maximal admissible face} if $F$ is not a subface of some
    other admissible face of $\sproj$.
    The same definitions apply if we replace $\sproj$ with
    $\qproj$, $\scone$ or $\qcone$.
\end{definition}

There are always admissible points in $\sproj$ (for instance,
scaled multiples of the vertex links in the underlying triangulation).
Likewise, there are always admissible points in the cones $\scone$ and
$\qcone$ (the origin, for example).  However, it might be the case that the
quadrilateral solution space $\qproj$ has no admissible points at all,
in which case the empty face becomes the unique maximal admissible
face of $\qproj$.

In general, faces of a polytope are simpler to deal with than arbitrary
sets of points---they have convenient representations
(such as intersections with supporting hyperplanes) and useful
combinatorial properties (which we discuss shortly).
Our first result is to show that, in each solution space,
the admissible faces together hold all of the admissible points.
Jaco and Oertel make a similar remark in \cite{jaco84-haken},
at the point where they introduce the projective solution space.

\begin{lemma}
    Every admissible point within the standard projective
    solution space $\sproj$ belongs to some admissible face of $\sproj$.
    The same is true if we replace $\sproj$ with
    $\qproj$, $\scone$ or $\qcone$.
\end{lemma}

\begin{proof}
    We work with $\sproj$ only; the arguments for $\qproj$,
    $\scone$ and $\qcone$ are identical.
    Let $\mathbf{p} \in \sproj$ be any admissible point, and let
    $F$ be the minimal-dimensional face of $\sproj$ containing $\mathbf{p}$
    (we can construct $F$ by taking the intersection of all faces
    containing $\mathbf{p}$).

    We claim that $F$ is an admissible face.  If not, let $\mathbf{q} \in F$
    be some inadmissible point in $F$.  Because $\mathbf{p}$ is
    admissible but $\mathbf{q}$ is not, there must be some coordinate
    position $i$ for which $p_i = 0$ and $q_i > 0$.

    Consider now the hyperplane $H = \{\mathbf{x} \in \R^{7n}\,|\,x_i=0\}$.
    It is clear that $H$ is a supporting hyperplane for $\sproj$ and
    that $\mathbf{p} \in H$ but $\mathbf{q} \notin H$.
    It follows that $F \cap H$ is a strict subface of $F$ containing our
    original point $\mathbf{p}$, contradicting the minimality of $F$.
\end{proof}

Because polyhedra have finitely many faces, every admissible face must
belong to some maximal admissible face.  This gives us the following
immediate corollary:

\begin{corollary} \label{c-face-union}
    The set of all admissible points in $\sproj$ is precisely the union of all
    maximal admissible faces of $\sproj$.
    The same is true if we replace $\sproj$ with
    $\qproj$, $\scone$ or $\qcone$.
\end{corollary}

\begin{remarks}
    It should be noted that this union of maximal admissible faces
    is generally not convex.  This means that we cannot (easily) apply
    the theory of convex polytopes to the ``admissible region'' within
    $\sproj$, which causes difficulties both for theoretical analysis (as in
    this paper) and for practical algorithms (see \cite{burton10-dd} for a
    detailed discussion).
    The maximal admissible faces are the largest admissible regions
    that \emph{can} be described as convex polytopes, and our strategy in
    Sections~\ref{s-quad}
    and~\ref{s-std} of this paper is to work with each maximal admissible
    face one at a time.

    It should also be noted that there may be faces of $\sproj$ that
    are \emph{not} admissible faces, but which contain admissible points.
    In particular, $\sproj$ itself is such a face.  We also see this
    in lower dimensions; for instance, $\sproj$ might have a
    non-admissible edge whose endpoints are both admissible vertices.
\end{remarks}

We turn our attention now to the face lattices of the various solution
spaces, and the structures formed by the admissible faces within them.

\begin{definition}[Admissible face semilattice]
    Let $P$ represent one of the solution spaces $\sproj$, $\qproj$,
    $\scone$ or $\qcone$.  The \emph{admissible face semilattice} of $P$,
    denoted $\lattadm(P)$, is the poset consisting of all
    admissible faces of $P$, ordered again by the subface relation.
\end{definition}

The use of the word ``semilattice'' will be justified shortly.
In the meantime,
it is clear that the admissible face semilattice $\lattadm(P)$ is a
substructure of the face lattice $\latt(P)$.
Figure~\ref{fig-lattice-adm} illustrates this for the quadrilateral
projective solution space, showing both $\latt(\qproj)$ and
$\lattadm(\qproj)$ for a three-tetrahedron triangulation\footnote{%
    The precise triangulation is described by the dehydration string
    \texttt{dafbcccxaqh}, using the notation of Callahan, Hildebrand and Weeks
    \cite{callahan99-cuspedcensus}.}
of the product space $\R P^2 \times S^1$.
The full face lattice is shown in grey, and the
admissible face semilattice is highlighted in black.
The admissible face semilattice contains
one maximal admissible edge, two maximal admissible vertices,
and no other maximal admissible faces at all.

\begin{figure}[htb]
    \centering
    \includegraphics[scale=0.8]{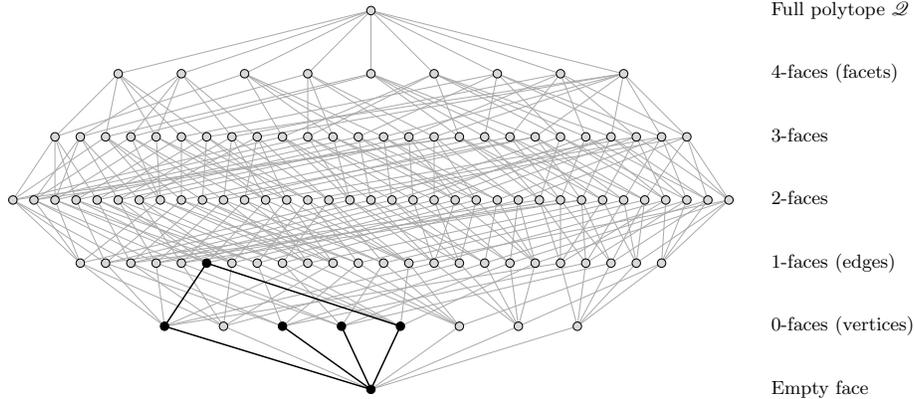}
    \caption{The face lattice and admissible face semilattice
        for an example triangulation}
    \label{fig-lattice-adm}
\end{figure}

One striking observation from Figure~\ref{fig-lattice-adm} is how few
admissible faces there are in comparison to the size of the full face lattice.
This is a pervasive phenomenon in normal surface theory,
and it highlights the importance of incorporating admissibility into
enumeration algorithms and complexity bounds.

The admissible face semilattice retains several key properties of the
face lattice, which we outline in the following lemma.
For this result we use interval notation: in a poset $S$
with elements $x \leq y$, the notation $[x,y]$ denotes the
\emph{interval} $\{w \in S\,|\,x \leq w \leq y\}$.

\begin{lemma} \label{l-intervals}
    The admissible face semilattice $\lattadm(\sproj)$ is the
    union of all intervals $[\emptyset,F]$ in the face lattice $\latt(\sproj)$,
    where $F$ ranges over all maximal admissible faces of $\sproj$.

    Every pair of faces $F,G \in \lattadm(\sproj)$ has a
    meet (i.e., a unique greatest lower bound),
    and $\lattadm(\sproj)$ has a unique minimal element (the empty face).
    The rank function of the face lattice $r\co \latt(\sproj) \to \N$
    maintains its covering property when restricted to $\lattadm(\sproj)$;
    that is, whenever $G$ covers $F$ in the poset $\lattadm(\sproj)$,
    we have $r(G) = r(F) + 1$.

    All of these results remain true if we replace $\sproj$ with
    $\qproj$, $\scone$ or $\qcone$.
\end{lemma}

\begin{proof}
    The fact that $\lattadm(\sproj)$ is the union of intervals
    $[\emptyset,F]$ for all maximal admissible faces $F$ follows immediately
    from Corollary~\ref{c-face-union}.
    The remaining observations follow from the properties of
    the face lattice $\latt(\sproj)$ and the observation that,
    for any face $F \in \lattadm(\sproj)$, all subfaces of $F$ are also
    in $\lattadm(\sproj)$.
    The arguments are identical for $\qproj$, $\scone$ and $\qcone$.
\end{proof}

The poset $\lattadm(\sproj)$ is generally not a lattice, since
joins $F \vee G$ need not exist.
Because meets exist however, $\lattadm(\sproj)$ is a
\emph{meet-semilattice} (and likewise for
$\qproj$, $\scone$ and $\qcone$); see \cite{stanley97-vol1} for details.

Throughout this section we work in all four solution spaces
$\sproj$, $\qproj$, $\scone$ and $\qcone$.  However, the cones
$\scone$ and $\qcone$ are precisely
the cones over the projective solution spaces $\sproj$ and $\qproj$,
and so their facial structures are tightly related.
The following result formalises this relationship, allowing us to
transport results between different spaces where necessary.

\begin{lemma} \label{l-proj-cone}
    Consider the cone map $F \mapsto \cone{F}$ from faces of
    $\sproj$ into the cone $\scone$.  This cone map satisfies
    all of the properties described in Lemma~\ref{l-prelim-cone};
    in particular, $F \mapsto \cone{F}$ is a bijection between the faces of
    $\sproj$ and the non-empty faces of $\scone$.

    Moreover, this bijection and its inverse both preserve admissibility.
    In other words, $\cone{F}$ is an admissible face of $\scone$ if and only
    if $F$ is an admissible face of $\sproj$.
    This means that the cone map is also a bijection between
    the admissible faces of $\sproj$ and the non-empty admissible faces
    of $\scone$, and a bijection between the maximal admissible faces of
    $\sproj$ and the maximal admissible faces of $\scone$.

    All of these results remain true if we replace $\sproj$ and $\scone$ with
    $\qproj$ and $\qcone$ respectively.
\end{lemma}

\begin{proof}
    We are able to use Lemma~\ref{l-prelim-cone} because
    $\sproj$ lies entirely within the projective hyperplane
    $\sum x_i = 1$, and so the origin lies outside the affine hull of $\sproj$.
    It is simple to show that the bijection $F \mapsto \cone{F}$
    and its inverse preserve
    admissibility: any inadmissible point in $F$ is also an
    inadmissible point in $\cone{F}$, and if $\mathbf{x}$ is an
    inadmissible point in $\cone{F}$ then $\mathbf{x}/\sum x_i$ is an
    inadmissible point in $F$.
    The remaining claims follow immediately from Lemma~\ref{l-prelim-cone}.
\end{proof}

One consequence of Lemma~\ref{l-prelim-cone} is that the face lattice
of $\scone$ is ``almost isomorphic'' to the face lattice of $\sproj$;
the only difference is that $\latt(\scone)$ contains one new element
(the empty face) that is dominated by all others.
What Lemma~\ref{l-proj-cone} shows is that the same relationship exists
between the admissible face semilattices.

From here we turn our attention to admissible faces and compatible pairs
of points.  Throughout the remainder of this section we explore the
relationships between these two concepts, culminating in
Corollary~\ref{c-maxadm-vertices} which categorises maximal admissible
faces in terms of pairwise compatible points and vertices.

\begin{lemma} \label{l-adm-face-to-compatible}
    Let $F$ be an admissible face of $\sproj$, $\qproj$,
    $\scone$ or $\qcone$.  Then any two points in $F$ are compatible.
\end{lemma}

\begin{proof}
    Suppose that $F$ contains two incompatible points
    $\mathbf{x},\mathbf{y}$.
    Because $\mathbf{x}$ and $\mathbf{y}$ are admissible but incompatible,
    their sum $\mathbf{x}+\mathbf{y}$ must have non-zero
    entries in the coordinate positions
    for two distinct quadrilateral types
    within the same tetrahedron.  Therefore the midpoint
    $\mathbf{z} = (\mathbf{x}+\mathbf{y})/2$ is inadmissible,
    contradicting the admissibility of the face $F$.
\end{proof}

From this result we obtain a simple but useful bound on the complexity of
admissible faces within our solution spaces.  Note that
by a ``facet'' of some $i$-face $F$, we mean an $(i-1)$-dimensional
subface of $F$.

\begin{corollary} \label{c-adm-facets}
    Every admissible face of $\qproj$ or $\qcone$ has at most $n$ facets,
    and every admissible face of $\sproj$ or $\scone$ has at most $5n$ facets.
\end{corollary}

\begin{proof}
    Let $F$ be an admissible face of $\qcone$.  Because any two points
    in $F$ are compatible (Lemma~\ref{l-adm-face-to-compatible}),
    it follows that for each tetrahedron of the underlying triangulation,
    two of the three corresponding
    quadrilateral coordinates are \emph{simultaneously} zero for all
    points in $F$.  In other words, $F$ lies within
    $2n$ distinct hyperplanes of the form $x_i = 0$ (and possibly more).

    Recall that $\qcone$ is the intersection of $\R^{3n}$ with the
    hyperplanes defined by the matching equations and the $3n$
    half-spaces defined by the inequalities $x_i \geq 0$.
    Because $F$ is the intersection of $\qcone$ with a supporting hyperplane,
    the argument above shows that $F$ is precisely the
    intersection of $\R^{3n}$ with some number of hyperplanes and
    at most $3n-2n=n$ half-spaces of the form $x_i \geq 0$.

    It is a standard result of polytope theory \cite{ziegler95} that the
    number of half-spaces in any representation of a polytope is at least the
    number of facets, whereupon the number of facets of $F$ can be at
    most $n$.

    The corresponding result in $\qproj$ is immediate from
    Lemma~\ref{l-proj-cone}, and the corresponding arguments in
    $\scone,\sproj \subseteq \R^{7n}$ show that $F$ has at most
    $7n-2n=5n$ facets instead.
\end{proof}

In Lemma~\ref{l-adm-face-to-compatible} we showed that every admissible
face must be filled with pairwise compatible points.  In the following
result we turn this around, showing that any set of pairwise compatible
points must belong to some maximal admissible face.

\begin{lemma} \label{l-adm-compatible-to-face}
    Let $X \subseteq \sproj$ be any set of admissible points in which every
    two points are compatible.  Then there is some maximal admissible face $F$
    of $\sproj$ for which $X \subseteq F$.
    The same is true if we replace $\sproj$ with
    $\qproj$, $\scone$ or $\qcone$.
\end{lemma}

\begin{proof}
    We consider the case $X \subseteq \sproj$; again the arguments for
    $\qproj$, $\scone$ and $\qcone$ are identical.
    As in the proof of Corollary~\ref{c-adm-facets}, the pairwise
    compatibility constraint shows that, for each tetrahedron
    of the underlying triangulation, two of the three corresponding
    quadrilateral coordinates are simultaneously zero for all
    points in $X$.  As a consequence, $X$ lies within
    all $2n$ corresponding hyperplanes of the form $x_i = 0$.

    Let $G$ be the intersection of $\sproj$ with these $2n$
    hyperplanes.  It follows that every point in $G$ is admissible,
    and that $X \subseteq G \subseteq \sproj$.
    Moreover, because each hyperplane $x_i = 0$ is a supporting hyperplane for
    $\sproj$, it follows that $G$ is a face of $\sproj$ (and therefore an
    admissible face).  By finiteness of the face lattice,
    the admissible face $G$ must in turn belong to some
    maximal admissible face $F$ containing all of the points in $X$.
\end{proof}

Note that the set $X$ might be contained in several distinct maximal
admissible faces.  However, there is always a unique admissible face
of minimal dimension containing $X$ (specifically, the intersection of
\emph{all} admissible faces containing $X$).

We come now to our categorisation of maximal admissible faces.
Lemma~\ref{l-maxadm-cat} gives necessary and sufficient conditions for
a face to be a maximal admissible face, and
Corollary~\ref{c-maxadm-vertices} extends these to necessary and
sufficient conditions for an arbitrary set of points.

\begin{lemma} \label{l-maxadm-cat}
    Let $F$ be any admissible face of the projective solution space $\sproj$.
    Then the following conditions are equivalent:
    \begin{enumerate}[(i)]
        \item \label{en-maxadm-cat-max}
        $F$ is a maximal admissible face of $\sproj$;
        \item \label{en-maxadm-cat-point}
        there is no admissible point in $\sproj$ that is not in $F$
        but that is compatible with every point in $F$;
        \item \label{en-maxadm-cat-vertex}
        there is no admissible vertex of $\sproj$ that is not in $F$
        but that is compatible with every vertex of $F$.
    \end{enumerate}
    The same is true if we replace $\sproj$ with $\qproj$.
    In the solution cones $\scone$ and $\qcone$, conditions
    (\ref{en-maxadm-cat-max}) and
    (\ref{en-maxadm-cat-point}) are equivalent but we cannot use
    (\ref{en-maxadm-cat-vertex}).
\end{lemma}

\begin{proof}
    We first prove
    (\ref{en-maxadm-cat-max}) $\Leftrightarrow$
    (\ref{en-maxadm-cat-point}) for all four solution spaces.
    As usual we work in $\sproj$ only, since the arguments in the other
    solution spaces are identical.

    For (\ref{en-maxadm-cat-max}) $\Rightarrow$ (\ref{en-maxadm-cat-point}),
    suppose that $F$ is a maximal admissible face and there is some
    admissible point $\mathbf{x} \in \sproj \backslash F$ compatible
    with every point in $F$.  Then by Lemma~\ref{l-adm-compatible-to-face}
    there is some admissible face containing $F \cup \{\mathbf{x}\}$,
    contradicting the maximality of $F$.

    For (\ref{en-maxadm-cat-point}) $\Rightarrow$ (\ref{en-maxadm-cat-max}),
    suppose that $F$ is not a maximal admissible face.  This means
    that there is some larger admissible face $G \supset F$,
    and from Lemma~\ref{l-adm-face-to-compatible} it follows that there
    is some point $\mathbf{x} \in G \backslash F$ that is admissible
    and compatible with every point in $F$.

    To prove (\ref{en-maxadm-cat-point}) $\Leftrightarrow$
    (\ref{en-maxadm-cat-vertex}) we require the additional fact that $\sproj$
    (or $\qproj$) is a polytope, which means that every face is the
    convex hull of its vertices.  This is why
    condition~(\ref{en-maxadm-cat-vertex}) fails in the cones
    $\scone$ and $\qcone$, where the only vertex is the origin.

    For (\ref{en-maxadm-cat-max}) $\Rightarrow$ (\ref{en-maxadm-cat-vertex}),
    suppose that $F$ is a maximal admissible face with vertex set $V$,
    and suppose there is some
    admissible vertex $\mathbf{u}$ of $\sproj$ not in $F$ but compatible
    with every $\mathbf{v} \in V$.  By Lemma~\ref{l-adm-compatible-to-face}
    there is some admissible face $G$ containing $V \cup \{\mathbf{u}\}$,
    and by convexity of faces it follows that $G \supseteq \conv(V) = F$.
    Because $\mathbf{u} \notin F$ we have $G \neq F$, contradicting the
    maximality of $F$.

    For (\ref{en-maxadm-cat-vertex}) $\Rightarrow$ (\ref{en-maxadm-cat-max}),
    suppose that $F$ is not a maximal admissible face; again there must be
    some larger admissible face $G \supset F$.  Because faces are
    convex hulls of their vertices, $G$ must contain
    some admissible vertex $\mathbf{v}$ not in $F$, and applying
    Lemma~\ref{l-adm-face-to-compatible} again we find that $\mathbf{v}$ is
    an admissible vertex of $\sproj$ not in $F$ but compatible with
    every vertex of $F$.
\end{proof}

We digress briefly to make a simple observation based on
Lemma~\ref{l-maxadm-cat}.  Recall the \emph{vertex links} from
Section~\ref{s-prelim-tri}, which correspond to normal surfaces
that surround the vertices of the triangulation $\tri$ and
consist entirely of triangular discs.

\begin{corollary} \label{c-maxadm-links}
    In the standard solution cone $\scone$,
    every maximal admissible face contains every vertex link from the
    underlying triangulation.
\end{corollary}

\begin{proof}
    Vertex links represent surfaces with only triangular discs, and so the
    corresponding vectors in $\R^{7n}$ do not contain any non-zero
    quadrilateral coordinates at all.  Therefore every vertex link
    is admissible and compatible with
    \emph{every} point $\mathbf{x} \in \scone$,
    and so by Lemma~\ref{l-maxadm-cat} every vertex link must belong to every
    maximal admissible face of $\scone$.
\end{proof}

It should be noted
that Corollary~\ref{c-maxadm-links} extends to the standard
projective solution space $\sproj$ if we replace each vertex link
$\mathbf{v}$ with the scaled multiple $\mathbf{v}/\sum v_i$.
However, it does not extend to the quadrilateral projective solution
space $\qproj$, since in quadrilateral coordinates every vertex link
projects to the zero vector.

For our final result of this section, we extend the categorisation of
Lemma~\ref{l-maxadm-cat} to apply to arbitrary sets of points within the
solution spaces.

\begin{corollary} \label{c-maxadm-vertices}
    Let $X \subseteq \sproj$ be any set of points.
    Then the following conditions are equivalent:
    \begin{enumerate}[(i)]
        \item \label{en-maxadm-sets-max}
        $X$ is a maximal admissible face of $\sproj$;
        \item \label{en-maxadm-sets-points}
        $X$ is a maximal set of admissible and
        pairwise compatible points in $\sproj$;
        \item \label{en-maxadm-sets-vertices}
        $X$ is the convex hull of a maximal set of admissible and
        pairwise compatible vertices of $\sproj$.
    \end{enumerate}

    In conditions~(\ref{en-maxadm-sets-points})
    and~(\ref{en-maxadm-sets-vertices}), ``maximal'' is used in
    the context of set inclusion.  For instance, in
    (\ref{en-maxadm-sets-points}) it means that there is no larger set
    $X' \supset X$ of admissible and pairwise compatible points in $\sproj$.

    These equivalences remain true if we replace $\sproj$ with $\qproj$.
    In the solution cones $\scone$ and $\qcone$, conditions
    (\ref{en-maxadm-cat-max}) and
    (\ref{en-maxadm-cat-point}) are equivalent but again we cannot use
    (\ref{en-maxadm-cat-vertex}).
\end{corollary}

\begin{proof}
    Steps (\ref{en-maxadm-sets-max}) $\Rightarrow$
    (\ref{en-maxadm-sets-points}) and
    (\ref{en-maxadm-sets-max}) $\Rightarrow$
    (\ref{en-maxadm-sets-vertices})
    follow immediately from Lemma~\ref{l-maxadm-cat}.
    To prove the remaining steps
    (\ref{en-maxadm-sets-points}) $\Rightarrow$
    (\ref{en-maxadm-sets-max}) and
    (\ref{en-maxadm-sets-vertices}) $\Rightarrow$
    (\ref{en-maxadm-sets-max}) we work in $\sproj$ as always,
    since the arguments are identical for $\qproj$, and also
    $\scone$ and $\qcone$ where applicable.

    For (\ref{en-maxadm-sets-points}) $\Rightarrow$
    (\ref{en-maxadm-sets-max}), let $X$ be some maximal set of
    admissible and pairwise compatible points in $\sproj$.
    By Lemma~\ref{l-adm-compatible-to-face} there is some maximal
    admissible face $F \supseteq X$, and if $F \neq X$ then
    Lemma~\ref{l-adm-face-to-compatible} contradicts the maximality of
    our original set $X$.

    For (\ref{en-maxadm-sets-vertices}) $\Rightarrow$
    (\ref{en-maxadm-sets-max}), let $X = \conv(V)$ where $V$ is a
    maximal set of admissible and pairwise compatible vertices of
    $\sproj$.  Again Lemma~\ref{l-adm-compatible-to-face} gives some
    maximal admissible face $F \supseteq V$.   Because $F$ is the convex
    hull of its vertices, if $F \neq X$ then $F$ must have some
    additional vertex $\mathbf{v} \notin V$.
    By Lemma~\ref{l-adm-face-to-compatible} it follows that
    $\mathbf{v}$ is admissible
    and compatible with every vertex in $V$, contradicting the
    maximality of $V$.
\end{proof}

\section{Bounds for general polytopes} \label{s-ubt}

Our ultimate aim is to place bounds on the complexity of the
\emph{admissible} face semilattice for the projective solution space.
To do this, we must first understand the complexity of
the \emph{full} face lattice for an arbitrary polytope.

We begin this section by examining the behaviour of McMullen's
upper bound as we change the number of facets $k$ (Lemma~\ref{l-ubt-k})
and the dimension $d$ (Lemma~\ref{l-ubt-d}).  We follow
with an asymptotic summation result that will prove useful in
later sections (Corollary~\ref{c-m-sums}).

\begin{notation}
    For any integers $2 \leq d < k$, let $M_{d,k}$ denote McMullen's
    upper bound as expressed in Theorem~\ref{t-ubt}:
    \[ M_{d,k} =
        \binom{k - \lfloor \frac{d+1}{2} \rfloor}{k - d} +
        \binom{k - \lfloor \frac{d+2}{2} \rfloor}{k - d}. \]
    A simple rearrangement gives the equivalent expression:
    \begin{equation} \label{eqn-m}
        M_{d,k} = \left\{
            \begin{array}{ll}
                \binom{k-\frac{d}{2}}{\frac{d}{2}} +
                \binom{k-\frac{d}{2}-1}{\frac{d}{2}-1} &
                    \mbox{if $d$ is even;} \smallskip \\
                2 \binom{k-\frac{d+1}{2}}{\frac{d+1}{2}-1} &
                    \mbox{if $d$ is odd.} \\
            \end{array} \right.
    \end{equation}
\end{notation}

Our first simple result describes the behaviour of $M_{d,k}$ as we vary
the number of facets.

\begin{lemma} \label{l-ubt-k}
    For any integers $2 \leq d < k < k'$, we have $M_{d,k} < M_{d,k'}$.
    That is, increasing the number of facets of a polytope
    will always increase McMullen's upper bound.
\end{lemma}

\begin{proof}
    This follows immediately from equation~(\ref{eqn-m}),
    using the relations $\binom{m}{i} < \binom{m+1}{i}$ for $1 \leq i \leq m$
    and $\binom{m}{0} = \binom{m+1}{0}$ for $0 \leq m$.
\end{proof}

Varying the dimension is a little more complicated.  McMullen's
bound is not a monotonic function of $d$, and in general there can be
many local maxima and minima as $d$ ranges from $2$ to $k-1$;
Figure~\ref{fig-ubt} illustrates this for $k=100$ facets.
However, $M_{d,k}$ is well-behaved for $d \leq k/2$, as shown by the
following result.

\begin{figure}[htb]
    \centering
    \includegraphics[scale=0.6]{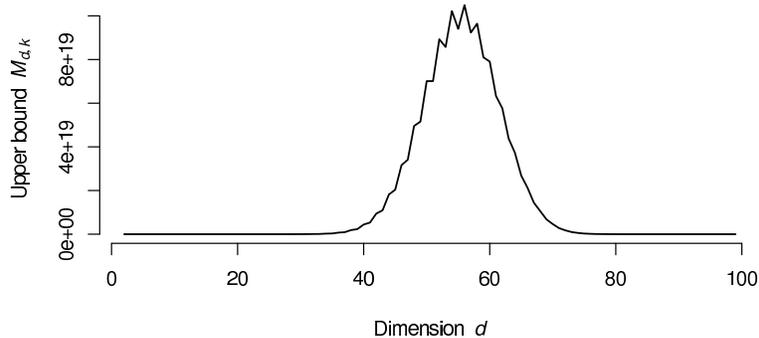}
    \caption{McMullen's upper bound $M_{d,k}$ for $k=100$ facets}
    \label{fig-ubt}
\end{figure}

\begin{lemma} \label{l-ubt-d}
    For any integers $d,k$ with $2 \leq d \leq k/2$, we have
    $M_{d,k} \leq M_{d+1,k}$.
    That is, increasing the dimension of a polytope will not
    decrease McMullen's upper bound, as long as there are sufficiently
    many facets.
\end{lemma}

\begin{proof}
    We begin by noting that $2 \leq d \leq k/2$ implies $d+1 < k$,
    so both $M_{d,k}$ and $M_{d+1,k}$ are defined.
    Our proof relies on a straightforward expansion of the
    binomial coefficients in equation~(\ref{eqn-m}).
    As with equation~(\ref{eqn-m}), we treat even and odd $d$ separately.

    If $d$ is even, let $d=2s$.  Then $M_{d,k} \leq M_{d+1,k}$ expands to
    $\binom{k-s}{s} + \binom{k-s-1}{s-1} \leq 2 \binom{k-s-1}{s}$,
    or
    \[ \frac{(k-s)!}{s!(k-2s)!} + \frac{(k-s-1)!}{(s-1)!(k-2s)!} \leq
        \frac{2(k-s-1)!}{s!(k-2s-1)!}. \]
    Cancelling common factors reduces this to $(k-s) + s \leq 2(k-2s)$;
    that is, $4s \leq k$, which is immediate from our initial
    condition $d \leq k/2$.

    If $d$ is odd, let $d=2s-1$.  Now $M_{d,k} \leq M_{d+1,k}$ expands to
    $2 \binom{k-s}{s-1} \leq \binom{k-s}{s} + \binom{k-s-1}{s-1}$, or
    \[ \frac{2(k-s)!}{(s-1)!(k-2s+1)!} \leq
        \frac{(k-s)!}{s!(k-2s)!} + \frac{(k-s-1)!}{(s-1)!(k-2s)!}. \]
    This simplifies to
    $2(k-s)s \leq (k-s)(k-2s+1) + s(k-2s+1)$, which in turn can be
    rearranged to $k^2 - k \leq 2(k-s)^2$.

    The odd case therefore gives $M_{d,k} \leq M_{d+1,k}$ if and only if
    $k^2 - k \leq 2(k-s)^2$, and again we prove this latter inequality
    from our initial conditions.  Using $2 \leq d \leq k/2$
    we obtain $s \leq (k+2)/4$, and so $k-s \geq (3k-2)/4 > 0$.
    From this we obtain
    \[ 2(k-s)^2 \geq 2\left(\frac{3k-2}{4}\right)^2
        = k^2 - k + \frac18(k-2)^2 \geq k^2 - k, \]
    and the result $M_{d,k} \leq M_{d+1,k}$ is established.
\end{proof}

We finish this section by studying sums of the form
$\sum_d \alpha^d M_{d,k}$; these sums reappear in
sections~\ref{s-quad} and~\ref{s-std} of this paper.
Our focus is on the asymptotic growth of these sums as
a function of $k$.  We approach this by first examining
the binomial coefficients $\binom{m-i}{i}$, and then returning to the
sums $\sum_d \alpha^d M_{d,k}$ in Corollary~\ref{c-m-sums}.

\begin{lemma} \label{l-ubt-binom-sums}
    For any integer $m \geq 0$ and any real $\alpha > 0$,
    define
    \[ S_\alpha(m) = \sum_{i=0}^{\lfloor m/2 \rfloor}
        \alpha^i \binom{m-i}{i}. \]
    Then $S_\alpha$ satisfies the recurrence relation
    $S_\alpha(m) = S_\alpha(m-1) + \alpha S_\alpha(m-2)$ for all $m \geq 2$,
    and the asymptotic growth rate of $S_\alpha$ relative to $m$ is
    \[ S_\alpha(m) \in \Theta\left( \left[
        \frac{1+\sqrt{1+4\alpha}}{2} \right]^m \right) .\]
\end{lemma}

\begin{proof}
    First we note that $S_\alpha(m)$ can be written as a sum over all
    $i \in \Z$, since $\binom{m-i}{i} = 0$ whenever $i < 0$ or
    $i > \lfloor m/2 \rfloor$.  Using the identity
    $\binom{m-i}{i} = \binom{m-i-1}{i}+\binom{m-i-1}{i-1}$, we have
    \begin{align*}
        S_\alpha(m) = \sum_{i\in\Z} \alpha^i \binom{m-i}{i}
                   &= \sum_{i\in\Z} \alpha^i \binom{m-i-1}{i} +
                      \sum_{i\in\Z} \alpha^i \binom{m-i-1}{i-1} \\
                   &= \sum_{i\in\Z} \alpha^i \binom{(m-1)-i}{i} +
                      \alpha \sum_{i\in\Z} \alpha^{i-1}
                      \binom{(m-2)-(i-1)}{i-1} \\
                   &= S_\alpha(m-1) + \alpha S_\alpha(m-2),
    \end{align*}
    thereby establishing our recurrence relation.

    The characteristic equation for this recurrence is
    $x^2 - x - \alpha = 0$, with roots
    $r_1 = \frac{1-\sqrt{1+4\alpha}}{2}$ and
    $r_2 = \frac{1+\sqrt{1+4\alpha}}{2}$;
    it is clear that $r_1 < 0 < r_2$ and $0 < |r_1| < |r_2|$.
    It follows that $S_\alpha(m) = c_1 r_1^m + c_2 r_2^m$ for some
    non-zero coefficients $c_1,c_2$ depending only on $\alpha$,
    and that the growth rate of $S_\alpha(m)$ relative to $m$ is therefore
    \[ S_\alpha(m) \in \Theta(r_2^m) = \Theta\left( \left[
        \frac{1+\sqrt{1+4\alpha}}{2} \right]^m \right) .\]
\end{proof}

\begin{corollary} \label{c-m-sums}
    For any real $\alpha$ in the range $0 < \alpha \leq 1$,
    consider the sum $\sum_{d=2}^{k-1} \alpha^d M_{d,k}$ as a function of
    the integer $k > 2$.
    This sum has an asymptotic growth rate of
    \[ \sum_{d=2}^{k-1} \alpha^d M_{d,k} \in
        \Theta\left( \left[
        \frac{1+\sqrt{1+4\alpha^2}}{2} \right]^k \right). \]
\end{corollary}

\begin{proof}
    Using equation~(\ref{eqn-m}) and setting $d=2i$ or $d=2i-1$ for
    even or odd $d$ respectively, we obtain the following identity:
    \begin{align*}
        \sum_{d=2}^{k-1} \alpha^d M_{d,k} \quad = \quad&
            \sum_{\stackrel{2 \leq d < k}{d\ \textrm{even}}} \alpha^d \left[
                \binom{k-\frac{d}{2}}{\frac{d}{2}} +
                \binom{k-\frac{d}{2}-1}{\frac{d}{2}-1} \right]
            +
            2 \sum_{\stackrel{3 \leq d < k}{d\ \textrm{odd}}} \alpha^d
                \binom{k-\frac{d+1}{2}}{\frac{d+1}{2}-1} \\
        =\quad&
            \sum_{i = 1}^{\lfloor (k-1)/2 \rfloor}
                \alpha^{2i} \binom{k-i}{i} +
            \sum_{i = 1}^{\lfloor (k-1)/2 \rfloor}
                \alpha^{2i} \binom{k-i-1}{i-1} +
            2 \sum_{i=2}^{\lfloor k/2 \rfloor}
                \alpha^{2i-1} \binom{k-i}{i-1} \\
        =\quad&
            \sum_{i \in \Z}
                (\alpha^2)^i \binom{k-i}{i} - 1
                - \{ \alpha^k\ \mbox{if $k$ is even} \} \\
        & + \alpha^2 \sum_{i \in \Z}
                (\alpha^2)^{i-1} \binom{(k-2)-(i-1)}{i-1}
                - \{ \alpha^k\ \mbox{if $k$ is even} \} \\
        & + 2 \alpha \sum_{i \in \Z}
                (\alpha^2)^{i-1} \binom{(k-1)-(i-1)}{i-1}
                - 2 \alpha
                - \{ 2 \alpha^k\ \mbox{if $k$ is odd} \} \\
        =\quad&
            S_{\alpha^2}(k) + \alpha^2 S_{\alpha^2}(k-2) +
            2\alpha S_{\alpha^2}(k-1) - 2 \alpha^k - 2 \alpha - 1,
    \end{align*}
    where $S_{\alpha^2}(\cdot)$ is the function defined earlier
    in Lemma~\ref{l-ubt-binom-sums} (though note that the subscript
    is now squared).
    Because each $S_{\alpha^2}(k)$ is non-negative and $|\alpha| \leq 1$,
    it follows immediately from Lemma~\ref{l-ubt-binom-sums} that
    \[ \sum_{d=2}^{k-1} \alpha^d M_{d,k} \in
        \Theta\left( \left[
        \frac{1+\sqrt{1+4\alpha^2}}{2} \right]^k \right). \]
\end{proof}

\section{The quadrilateral solution set} \label{s-quad}

In this section we combine the structural results of
Section~\ref{s-lattice} with the asymptotic bounds of
Section~\ref{s-ubt} to yield our first main result:
a new bound on the size of the quadrilateral solution set.

Recall that the \emph{quadrilateral solution set} is the set of all
admissible vertices of the quadrilateral projective solution space $\qproj$.
Little is currently known about the size of this set; the only
theoretical bound to date is $4^n$, as outlined in Section~\ref{s-prelim-tri}.

In this paper we employ more sophisticated techniques to bring this
bound down to approximately $O(3.303^n)$.  Our broad strategy is as
follows.  We first bound the number of maximal admissible faces of
each dimension; in particular, we show that there are at most
$3^{n-1-d}$ maximal admissible faces of each dimension $d \leq n-1$,
and no maximal admissible faces of any dimension $d \geq n$.
We then convert these results into a bound on the number of
admissible vertices using McMullen's theorem and the asymptotic
results of Section~\ref{s-ubt}.

Throughout this section we denote the coordinates of a vector
$\mathbf{x} \in \R^{3n}$ by
\[ \mathbf{x} = ( x_{1,1},x_{1,2},x_{1,3},\ x_{2,1},x_{2,2},x_{2,3},
    \ \ldots,\ x_{n,1},x_{n,2},x_{n,3} ), \]
where $x_{i,j}$ is the coordinate representing the $j$th
quadrilateral type within the $i$th tetrahedron.
We also make repeated use of the tetrahedral
solutions $\tet{1},\ldots,\tet{n} \in \qcone$; recall from
Section~\ref{s-prelim-tri} that the
$k$th tetrahedral solution $\tet{k}$ has
$\tet{k}_{k,1}=\tet{k}_{k,2} = \tet{k}_{k,3} = 1$ and all $(3n-3)$
remaining coordinates set to zero.

\begin{lemma} \label{l-quad-max-dim}
    Every admissible face of the quadrilateral projective
    solution space has dimension $\leq n-1$.
\end{lemma}

\begin{proof}
    Let $F$ be some $d$-dimensional admissible face of the quadrilateral
    projective solution space $\qproj$, and let $\cone{F}$ be the corresponding
    $(d+1)$-dimensional admissible face of the quadrilateral
    solution cone $\qcone$.
    Every pair of points in $\cone{F}$ must be compatible
    (Lemma~\ref{l-adm-face-to-compatible}),
    and so for each $i=1,\ldots,n$ at least two
    of the three coordinates $x_{i,1},x_{i,2},x_{i,3}$ must be
    simultaneously zero for \emph{all} points $\mathbf{x} \in \cone{F}$.

    It follows that the entire face $\cone{F}$ lies within some $n$-dimensional
    subspace $S \subseteq \R^{3n}$ defined by setting $2n$ coordinates
    equal to zero.  We therefore have $\dim \cone{F} \leq \dim S$;
    that is, $d+1 \leq n$, or $d \leq n-1$.
\end{proof}

\begin{lemma} \label{l-quad-bound-dim}
    For each $d \in \{0,\ldots,n-1\}$,
    the number of maximal admissible faces of dimension $d$ in the
    quadrilateral projective solution space is at most $3^{n-1-d}$.
\end{lemma}

\begin{proof}
    Let $F_1,\ldots,F_k$ be distinct maximal admissible $d$-faces
    within the quadrilateral projective solution space
    $\qproj$, where $k > 3^{n-1-d}$.  For convenience we work in
    the quadrilateral solution cone $\qcone$ instead,
    using the corresponding maximal
    admissible faces $\cone{F_1},\ldots,\cone{F_k}$ each of dimension $d+1$.

    Our strategy is to construct a decreasing sequence of linear subspaces
    $\R^{3n} \supset S_0 \supset S_1 \supset \ldots \supset S_n$
    with the following properties:
    \begin{enumerate}[(i)]
        \item \label{en-seq-tet}
        Each subspace $S_i$ contains all of the tetrahedral solutions
        $\tet{i+1},\ldots,\tet{n}$.

        \item \label{en-seq-size}
        For each subspace $S_i$, there is some integer $t_i \geq 0$
        for which $S_i$ has dimension $\leq 2n-i-t_i$, and for which
        $S_i$ contains strictly more than $3^{n-1-d-t_i}$
        of the maximal admissible faces $\cone{F_1},\ldots,\cone{F_k}$.

        \item \label{en-seq-compat}
        For each subspace $S_i$ and each integer
        $j=1,\ldots,i$, the subspace $S_i$ is contained in at
        least two of the three hyperplanes
        $x_{j,1} = 0$, $x_{j,2} = 0$ and $x_{j,3} = 0$.
        In other words, for each of the first $i$ tetrahedra,
        at least two of the three corresponding quadrilateral
        coordinates are simultaneously zero for all points in $S_i$.
    \end{enumerate}
    We construct this sequence inductively as follows:
    \begin{itemize}
        \item We set the initial subspace $S_0$ to be the solution space
        to the quadrilateral matching equations.
        Property~(\ref{en-seq-tet}) holds because
        $\tet{1},\ldots,\tet{n} \in \qcone \subseteq S_0$.
        Property~(\ref{en-seq-size}) holds with $t_0 = 0$,
        since we have $\dim S_0 = 2n$ from Theorem~\ref{t-matching-dim},
        and since all $k > 3^{n-1-d}$ of our maximal admissible faces
        are contained within $\qcone \subseteq S_0$.
        Property~(\ref{en-seq-compat}) is vacuously satisfied for $i=0$.

        \item For each $i > 0$, we construct $S_i$ from $S_{i-1}$ as
        follows.  Let $X = \{ \cone{F_j}\,|\,\cone{F_j} \subseteq S_{i-1}\}$;
        that is, the set of all maximal admissible faces from our
        original collection that are contained within the previous
        subspace $S_{i-1}$.
        Because each $\cone{F_j}$ is an admissible face, we know
        from Lemma~\ref{l-adm-face-to-compatible}
        that each $\cone{F_j}$ lies in at least two (and possibly all three)
        of the hyperplanes $x_{i,1} = 0$, $x_{i,2} = 0$ and $x_{i,3} = 0$
        (though \emph{which} of these hyperplanes $\cone{F_j}$ belongs
        to will typically depend on $j$).
        Consider the following two cases:
        \begin{enumerate}[(a)]
            \item \label{en-case-same}
            Suppose that all $\cone{F_j} \in X$ are
            \emph{simultaneously} contained in at least two of the three
            hyperplanes $x_{i,1} = 0$, $x_{i,2} = 0$ and $x_{i,3} = 0$;
            that is, this choice does not depend on $j$.
            Without loss of generality, let these two hyperplanes be
            $x_{i,2} = 0$ and $x_{i,3} = 0$.

            In this case we let $S_i$ be the intersection of the
            subspace $S_{i-1}$ with the hyperplanes
            $x_{i,2} = 0$ and $x_{i,3} = 0$.
            Note that every face $\cone{F_j} \in X$ belongs to
            the subspace $S_i$ as a result.

            Property~(\ref{en-seq-tet}) holds for $S_i$ because each of the
            tetrahedral solutions $\tet{i+1},\ldots,\tet{n}$ belongs to
            $S_{i-1}$ as well as all three hyperplanes
            $x_{i,1} = 0$, $x_{i,2} = 0$ and $x_{i,3} = 0$.
            Property~(\ref{en-seq-compat}) for $S_i$ follows immediately from
            our construction.

            Property~(\ref{en-seq-size}) for $S_i$ is established as follows.
            Let $t_i = t_{i-1}$.
            We note that $S_i$ is a \emph{strict} subspace of $S_{i-1}$,
            because the tetrahedral solution $\tet{i}$ lies in $S_{i-1}$
            (from property~(\ref{en-seq-tet}) for $S_{i-1}$)
            but not $S_i$ (because $\tet{i}_{i,2},\tet{i}_{i,3} \neq 0$).
            It follows that $\dim S_i \leq \dim S_{i-1} - 1 \leq
            2n-(i-1)-t_{i-1}-1 = 2n-i-t_i$.
            Furthermore, our construction ensures that every face
            $\cone{F_j} \in X$ lies within $S_i$,
            and using property~(\ref{en-seq-size}) for $S_{i-1}$
            there are strictly more than
            $3^{n-1-d-t_{i-1}} = 3^{n-1-d-t_i}$ such faces.

            \smallskip

            \item Otherwise, all $\cone{F_j} \in X$ are
            not simultaneously contained in at least two of the three
            hyperplanes $x_{i,1} = 0$, $x_{i,2} = 0$ and $x_{i,3} = 0$.
            Consider the three sets
            \begin{align*}
                X_1 &= \{ \cone{F_j} \in X\,\left|
                    \ \cone{F_j}\ \mbox{lies in both hyperplanes}
                    \ x_{i,2}=0,\ x_{i,3}=0 \right.\}; \\
                X_2 &= \{ \cone{F_j} \in X\,\left|
                    \ \cone{F_j}\ \mbox{lies in both hyperplanes}
                    \ x_{i,3}=0,\ x_{i,1}=0 \right.\}; \\
                X_3 &= \{ \cone{F_j} \in X\,\left|
                    \ \cone{F_j}\ \mbox{lies in both hyperplanes}
                    \ x_{i,1}=0,\ x_{i,2}=0 \right.\}.
            \end{align*}
            We know from our earlier comments that $X = X_1 \cup X_2 \cup X_3$.
            Without loss of generality suppose that $X_1$ is the largest
            of these three sets; in particular, $|X_1| \geq |X| / 3$.

            For this case we define $S_i$ to be the intersection of the
            subspace $S_{i-1}$ and the two hyperplanes
            $x_{i,2}=0$ and $x_{i,3}=0$.  Note that the faces
            $\cone{F}_j$ that lie within $S_i$ are precisely those in
            the set $X_1$.

            Once again properties~(\ref{en-seq-tet}) and~(\ref{en-seq-compat})
            for $S_i$ are simple consequences of our construction.
            To establish property~(\ref{en-seq-size}) for $S_i$, we let
            $t_i = t_{i-1} + 1$.  The number of faces $\cone{F_j}$ in $S_i$
            is $|X_1| \geq |X|/3 > 3^{n-1-d-t_{i-1}}/3 = 3^{n-1-d-t_i}$
            as required.  Bounding the dimension of $S_i$ requires a
            little more work.

            We know that there is some face $\cone{F_a} \in X$ that is
            not in the set $X_1$ (otherwise we would have fallen back to
            case~(\ref{en-case-same})).
            However, this face $\cone{F_a}$ must belong
            to one of $X_1$, $X_2$ or $X_3$; without loss of generality
            suppose that $\cone{F_a} \in X_2$.  Let $S_i'$ be the
            intersection of the subspace $S_{i-1}$ with the hyperplane
            $x_{i,3}=0$.  Because $\tet{i} \in S_{i-1}$ but
            $\tet{i}_{i,3} \neq 0$ it follows that $S_i'$ is a strict
            subspace of $S_{i-1}$, and we have
            $\dim S_i' \leq \dim S_{i-1} - 1$.

            Now we find that $S_i$ is the intersection of $S_i'$ with the
            hyperplane $x_{i,2}=0$.  The face $\cone{F_a}$
            lies within the hyperplane $x_{i,3}=0$ and therefore lies
            in $S_i'$; however, because
            $\cone{F_a} \notin X_1$ it cannot also lie in the hyperplane
            $x_{i,2}=0$, which means that $\cone{F_a}$ does not lie in $S_i$.
            Therefore $S_i$ is a strict subspace of $S_i'$, and we have
            $\dim S_i \leq \dim S_i' - 1 \leq \dim S_{i-1} - 2$,
            giving a final dimension
            $\dim S_i \leq 2n-(i-1)-t_{i-1}-2 = 2n-i-t_i$.
        \end{enumerate}
    \end{itemize}

    This establishes properties~(\ref{en-seq-tet})--(\ref{en-seq-compat})
    for our sequence of linear subspaces
    $\R^{3n} \supset S_0 \supset S_1 \supset \ldots \supset S_n$.
    We finish our proof by considering the final subspace $S_n$.

    From property~(\ref{en-seq-size}) we know that $S_n$ contains at least
    one of the maximal admissible faces $\cone{F_1},\ldots,\cone{F_k}$,
    and so $\dim S_n \geq d+1$.  The dimension
    constraint of property~(\ref{en-seq-size}) then gives $t_n \leq n-1-d$,
    whereupon we find that $S_n$ contains
    strictly more than $3^{n-1-d-t_n} \geq 1$ of the maximal admissible faces
    $\cone{F_1},\ldots,\cone{F_k}$.  That is, $S_n$ must contain
    \emph{at least two} of these faces.  Let these faces be
    $\cone{F_a}$ and $\cone{F_b}$.

    By property~(\ref{en-seq-compat}) we know that all points in
    $S_n$ are pairwise compatible, and so every point in $\cone{F_a}$ must be
    compatible with every point in $\cone{F_b}$.  However, from
    Corollary~\ref{c-maxadm-vertices} we know that $\cone{F_a}$ and
    $\cone{F_b}$ are each maximal sets of admissible
    and pairwise compatible points in $\qcone$, giving
    $\cone{F_a} = \cone{F_b}$ and a contradiction.
\end{proof}

This bound of $\leq 3^{n-1-d}$ maximal admissible faces of dimension $d$
appears to be tight for large dimensions $d$ (in particular, for
$d \geq \frac{n}{2}-1$ as we discuss in Section~\ref{s-discussion}).
Nevertheless, even for large dimensions this not the entire
story.  We might be able to achieve equality for some large
dimensions $d$, but we cannot achieve equality for \emph{all} large dimensions
simultaneously, as indicated by the following result.

\begin{lemma} \label{l-quad-simultaneous}
    If the quadrilateral projective solution space has a maximal
    admissible face of dimension $n-1$, then this is the \emph{only}
    maximal admissible face (of any dimension).
\end{lemma}

\begin{proof}
    Suppose that we have two distinct maximal admissible faces
    $F,G \subseteq \qproj$ where $\dim F = n-1$.
    Once again we work in the quadrilateral solution cone $\qcone$,
    using the corresponding
    maximal admissible faces $\cone{F},\cone{G}$ with $\dim \cone{F} = n$.

    For each $i=1,\ldots,n$, Lemma~\ref{l-adm-face-to-compatible}
    shows that face $\cone{F}$ must lie within at least two of the
    three hyperplanes
    $x_{i,1} = 0$, $x_{i,2} = 0$ and $x_{i,3} = 0$.  Likewise,
    $\cone{G}$ must lie within at least two of these hyperplanes, and so
    both $\cone{F}$ and $\cone{G}$ must \emph{simultaneously} lie in
    at least one of the hyperplanes
    $x_{i,1} = 0$, $x_{i,2} = 0$ or $x_{i,3} = 0$.  Without loss of
    generality let this common hyperplane be $x_{i,1}=0$.

    Let $S$ be the solution space to the quadrilateral matching
    equations in $\R^{3n}$; by Theorem~\ref{t-matching-dim} we have
    $\dim S = 2n$.
    Let $S'$ be the subspace of $S$ formed by intersecting $S$ with each
    of the hyperplanes $x_{i,1}=0$ for $i=1,\ldots,n$.

    Each of the tetrahedral solutions $\tet{i}$ belongs to $S$ but not
    $S'$.  It is clear that the tetrahedral solutions are linearly
    independent (their non-zero coordinates appear in distinct
    positions), and so $\dim S' \leq \dim S - n = n$.  Faces
    $\cone{F}$ and $\cone{G}$ still lie within $S'$ however, and
    because $\dim \cone{F} = n$ it follows that $\dim S' = n$
    and that $S'$ is the affine hull of $\cone{F}$.

    We now see that the face $\cone{G}$ lies within the affine hull of
    the face $\cone{F}$; it follows that $\cone{G}$ must be a subface of
    $\cone{F}$, contradicting the maximality of $\cone{G}$.
\end{proof}

Lemmata~\ref{l-quad-max-dim} and~\ref{l-quad-bound-dim} together bound the
number of maximal admissible faces of every dimension in $\qproj$.
We can now use
these results to prove our main theorem, which is a new bound on the
size of the quadrilateral solution set (that is, the number of admissible
vertices of $\qproj$).

\begin{theorem} \label{t-quad-bound}
    The size of the quadrilateral solution set is asymptotically
    bounded above by
    \[ O\left(\left[\frac{3+\sqrt{13}}{2}\right]^n\right)
    \quad \simeq \quad O(3.303^n). \]
\end{theorem}

\begin{proof}
    Let $\kappa$ denote the number of admissible vertices of the
    quadrilateral projective solution space $\qproj$.
    Our strategy is to bound $\kappa$ by working through the maximal
    admissible faces of each dimension.
    To avoid small-case irregularities, we assume that $n \geq 3$.

    More specifically, each admissible vertex must belong to some
    maximal admissible face of dimension $\geq 0$.  We can therefore
    bound $\kappa$ by (i)~computing McMullen's bound for the number
    of vertices of each maximal admissible face, and then (ii)~summing
    these bounds over all maximal admissible faces of all dimensions.
    We might count some vertices multiple times in this sum,
    but each vertex will be counted at least once.

    We piece this sum together one dimension at a time, using
    Lemma~\ref{l-quad-bound-dim} to bound the number of maximal
    admissible $d$-faces for each $d$.
    \begin{itemize}
        \item There are $\leq 3^{n-1}$ maximal admissible $0$-faces,
        adding $3^{n-1}$ admissible vertices to our sum.
        \item There are $\leq 3^{n-2}$ maximal admissible $1$-faces,
        adding $2 \cdot 3^{n-2}$ admissible vertices to our sum
        (since each $1$-face is an edge, and has precisely two vertices).
        \item For each $d$ in the range $2 \leq d \leq n-1$,
        there are $\leq 3^{n-1-d}$ maximal admissible $d$-faces.
        Each of these $d$-faces has at most $n$ facets
        (Corollary~\ref{c-adm-facets})
        and therefore at most $M_{d,n}$ vertices (Theorem~\ref{t-ubt} and
        Lemma~\ref{l-ubt-k}).
        This adds $\leq 3^{n-1-d} \cdot M_{d,n}$ admissible vertices to
        our sum.
    \end{itemize}

    By Lemma~\ref{l-quad-max-dim} there are no admissible $d$-faces for
    any dimension $d \geq n$, and so our final bound on $\kappa$ becomes
    \begin{align*}
        \kappa \quad&\leq\quad 3^{n-1} + 2 \cdot 3^{n-2} +
            \sum_{d=2}^{n-1} 3^{n-1-d} \cdot M_{d,n} \\
        &=\quad 3^{n-1} + 2 \cdot 3^{n-2} +
            3^{n-1} \sum_{d=2}^{n-1} (1/3)^d \cdot M_{d,n} \\
        &\in\quad O\left(3^n + 3^n \cdot \left[
            \frac{1+\sqrt{1+4/9}}{2} \right]^n \right),
    \end{align*}
    using the asymptotic bound from Corollary~\ref{c-m-sums}.
    The second term in this final expression dominates the first, and we have
    \[ \kappa \in O\left(\left[3 \cdot
        \frac{1+\sqrt{13/9}}{2} \right]^n\right) =
        O\left(\left[\frac{3+\sqrt{13}}{2}\right]^n\right)
        \simeq O(3.303^n). \]
\end{proof}

\section{The standard solution set} \label{s-std}

Having established new bounds for the quadrilateral projective
solution space $\qproj \subseteq \R^{3n}$, we can now transport this
information to the standard projective solution space
$\sproj \subseteq \R^{7n}$.

As noted in the introduction,
the first upper bound on the number of admissible vertices of
$\sproj$ was $128^n$, proven by Hass et~al.\ \cite{hass99-knotnp}.
The best bound known to date is approximately $O(29.03^n)$,
proven by the author \cite{burton10-complexity}.
The argument by Hass et~al.\ relies on the fact that
each vertex can be described as an intersection of facets of
$\sproj$, and with $\leq 7n$ facets there can be at
most $2^{7n}=128^n$ such intersections.
The bound of
$O(29.03^n)$ was obtained by deriving a simple asymptotic extension to
McMullen's upper bound theorem.

In this paper we tighten the best upper bound in standard coordinates
to approximately $O(14.556^n)$ admissible vertices.
Our strategy is to draw on our earlier
results in quadrilateral coordinates.  We begin by describing a bijection
between maximal admissible faces of $\qproj$ and $\sproj$,
and then once again we aggregate over faces of varying dimensions.

As a further application of these techniques, we examine the
special but important case of a one-vertex triangulation.
The author sketches a proof in \cite{burton10-dd} that
for a one-vertex triangulation the solution space $\sproj$ has
approximately $O(15^n/\sqrt{n})$ vertices.
Our final result of this paper is to tighten this bound to
approximately $O(4.852^n)$.

\begin{lemma} \label{l-bijection-std-quad}
    Let $v$ be the number of vertices in the underlying triangulation
    $\tri$.  Then there is a bijection between the maximal admissible
    faces of $\qproj$ and the maximal admissible faces of $\sproj$
    that maps $i$-faces of $\qproj$ to $(i+v)$-faces of $\sproj$ for
    every $i$.
\end{lemma}

\begin{proof}
    For convenience we work in the solution
    cones $\scone$ and $\qcone$ instead of the projective solution
    spaces $\sproj$ and $\qproj$; Lemma~\ref{l-proj-cone} shows this
    formulation to be equivalent.
    We establish our bijection in the
    direction from $\scone$ to $\qcone$ using the (linear)
    quadrilateral projection map
    $\pi \co \R^{7n} \to \R^{3n}$.  Recall from
    Section~\ref{s-prelim-tri} that $\pi$ is an onto map that
    preserves admissibility and inadmissibility, as well as
    compatibility and incompatibility.

    We can apply the map $\pi$ to sets of points (and in particular,
    faces of $\scone$).  Let
    $\fpi(X)$ denote the image $\{ \pi(\mathbf{x})\ |\ \mathbf{x} \in X \}$
    for any set $X \subseteq \scone$.
    Although $\fpi$ might not map faces to faces in general,
    we claim that it does map \emph{maximal admissible faces} of $\scone$
    to maximal admissible faces of $\qcone$.  Moreover, we claim that
    $\fpi$ is in fact the bijection that we seek.  We prove these claims
    in stages.

    \begin{itemize}
        \item \emph{$\fpi$ maps maximal admissible faces of $\scone$
        to maximal admissible faces of $\qcone$.}

        Let $F$ be a maximal admissible face of $\scone$.  Because
        $\pi$ preserves admissibility and compatibility,
        all points in $\fpi(F)$ are admissible and pairwise compatible.
        It follows from
        Lemma~\ref{l-adm-compatible-to-face} that there is some
        maximal admissible
        face $G$ of $\qcone$ for which $\fpi(F) \subseteq G$.

        If $\fpi(F)$ is not itself a maximal admissible face then we can find
        some admissible point $\mathbf{g} \in G \backslash \fpi(F)$.
        We know that $\mathbf{g}$ is compatible with every point in $\fpi(F)$
        (Lemma~\ref{l-adm-face-to-compatible}),
        and because $\pi$ preserves inadmissibility and incompatibility
        it follows that every point in the preimage $\pi^{-1}(\mathbf{g})
        \subseteq \scone \backslash F$ is admissible and compatible with every
        point in $F$.
        This contradicts the assumption that $F$ is a maximal admissible
        face of $\scone$ (Corollary~\ref{c-maxadm-vertices}), and it
        follows that $\fpi(F)$ must indeed
        be a maximal admissible face of $\qcone$.

        \item \emph{As a map between maximal admissible faces,
        $\fpi$ is one-to-one.  That is, for every two distinct maximal
        admissible faces $F,G \subseteq \scone$, we have
        $\fpi(F) \neq \fpi(G)$.}

        Let $F$ and $G$ be distinct maximal admissible faces of $\scone$.
        By Corollary~\ref{c-maxadm-vertices} there
        exist admissible and \emph{incompatible} points
        $\mathbf{f} \in F$ and $\mathbf{g} \in G$.  Because $\pi$
        preserves incompatibility it follows that
        $\pi(\mathbf{f})$ and $\pi(\mathbf{g})$ are incompatible
        points in $\qcone$.
        That is, we have two incompatible points
        $\pi(\mathbf{f})$ and $\pi(\mathbf{g})$ in the maximal admissible
        faces $\fpi(F)$ and $\fpi(G)$ respectively,
        and from Corollary~\ref{c-maxadm-vertices} again it follows that
        $\fpi(F) \neq \fpi(G)$.

        \item \emph{As a map between maximal admissible faces,
        $\fpi$ is onto.  That is, for every maximal admissible face
        $G \subseteq \qcone$, there is a maximal admissible face
        $F \subseteq \scone$ for which $\fpi(F) = G$.}

        Let $G$ be any maximal admissible face of $\qcone$, and consider
        the preimage $\pi^{-1}(G)$.  Because $\pi$ preserves
        inadmissibility and incompatibility, $\pi^{-1}(G)$ must be a
        collection of admissible and pairwise compatible points in
        $\scone$.  By Lemma~\ref{l-adm-compatible-to-face} there is some
        maximal admissible face $F \subseteq \scone$
        for which $F \supseteq \pi^{-1}(G)$.
        This gives us $\fpi(F) \supseteq G$, and because both
        $\fpi(F)$ and $G$ are maximal admissible faces of $\qcone$ it
        follows that $\fpi(F) = G$.
    \end{itemize}

    This shows that $\fpi$ yields a bijection between the maximal
    admissible faces of $\scone$ and the maximal admissible faces of
    $\qcone$.
    All that remains now is to establish how $\fpi$ affects the dimensions
    of these faces.

    Let $F$ be some maximal admissible face in $\scone$.
    We know from Section~\ref{s-prelim-tri}
    that the kernel of the linear map $\pi$ is generated by the $v$
    linearly independent vertex links (where $v$ is the number of
    vertices in the underlying triangulation).  Moreover,
    Corollary~\ref{c-maxadm-links} shows that all $v$ vertex
    links belong to the maximal admissible face $F$.
    Therefore we must have $\dim{F} = \dim{\fpi(F)} + v$.
\end{proof}

It should be noted that $\qproj$ may contain no admissible points at
all; in this case $\qproj$ has a single maximal admissible face of
dimension $-1$ (the empty face).  In standard coordinates,
$\sproj$ will always have admissible points (in particular, we always
have the vertex links).

Now that we are equipped with this bijection, we aim to bound the
dimensions of the maximal admissible faces of $\sproj$.  To do this,
we must place a bound on the number of vertices $v$ of the underlying
triangulation.

\begin{lemma} \label{l-min-vert}
    Any closed and connected
    3-manifold triangulation with $n > 2$ tetrahedra can have
    at most $n+1$ vertices.
\end{lemma}

\begin{proof}
    Let $\tri$ be such a triangulation, and
    let $G$ denote the \emph{face pairing graph} of $\tri$.  This
    is the connected 4-valent multigraph
    whose vertices represent tetrahedra of $\tri$ and whose
    edges represent identifications between tetrahedron faces
    (in particular, loops and multiple edges are allowed).
    See \cite{burton04-facegraphs} for further discussion and
    explicit examples of face pairing graphs.\footnote{%
    $G$ can also be thought of as the
    \emph{dual 1-skeleton} of $\tri$, with a
    \emph{dual vertex} at the centre of every tetrahedron of $\tri$ and a
    \emph{dual edge} running through every face of $\tri$.}

    Let $S$ be a spanning tree within $G$, and let $\tri_S$ denote the
    ``partial triangulation'' constructed from the same $n$ tetrahedra
    by making \emph{only} the face identifications described by the
    edges of $S$.  This means that $\tri_S$ is a connected simplicial complex
    formed from $n$ tetrahedra by identifying precisely $n-1$ pairs of
    faces.  Moreover, the original triangulation $\tri$ can be obtained from
    $\tri_S$ by identifying the remaining $n+1$ pairs of faces that
    correspond to the edges of $G \backslash S$.
    Figure~\ref{fig-spanning} illustrates a face pairing graph $G$ with
    a spanning tree $S$, and shows how the partial triangulation
    $\tri_S$ might appear.

    \begin{figure}[htb]
        \centering
        \includegraphics{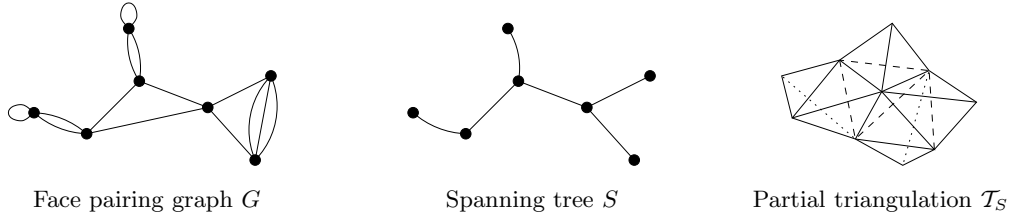}
        \caption{The partial triangulation $\tri_S$ corresponding to a
            spanning tree in $G$}
        \label{fig-spanning}
    \end{figure}

    Let $v$ and $v_S$ denote the number of vertices in $\tri$ and
    $\tri_S$ respectively.  It is clear that $v \leq v_S$, since we
    obtain $\tri$ from $\tri_S$ by making additional face
    identifications (which may identify vertices of $\tri_S$ together to
    reduce the total vertex count) but never adding new tetrahedra
    (and therefore never increasing the total vertex count).

    It is straightforward to count the number of vertices in $\tri_S$.
    Because $S$ is a spanning tree, we construct $\tri_S$ as follows:
    \begin{itemize}
        \item Begin with some initial tetrahedron $\Delta_1$,
        which gives us four initial vertices for $\tri_S$.
        \item Follow by joining some new tetrahedron $\Delta_2$ to
        $\Delta_1$ along a single face.  This introduces precisely one
        additional vertex to $\tri_S$, since the other three vertices of
        $\Delta_2$ (those on the joining face) become identified with
        the original vertices from $\Delta_1$.
        \item Next we join some new tetrahedron $\Delta_3$ to
        \emph{either} $\Delta_2$ or $\Delta_1$ along a single face.
        Again this introduces precisely one new vertex to $\tri_S$
        (the vertex of $\Delta_3$ not on the joining face).
        \item We continue this procedure, joining the remaining tetrahedra
        $\Delta_4,\ldots,\Delta_n$ into our structure along a single face
        each, creating one new vertex for $\tri_S$ every time.
    \end{itemize}
    It follows that the number of vertices in $\tri_S$ is precisely
    $v_S = n+3$, and we obtain $v \leq n+3$ as a result.

    We can reduce our bound from $n+3$ to $n+1$ by studying the
    \emph{leaves} of the tree $S$; that is, vertices of the tree with
    only one incident edge.  Each leaf corresponds to a tetrahedron of
    $\tri_S$ with only one face joined to the remainder of the structure.
    Moreover, the vertex opposite this face is not (yet) identified with any
    other vertices of any tetrahedron at all; we call this the
    \emph{isolated vertex} of the leaf.  This situation is illustrated
    in Figure~\ref{fig-leaf}.

    \begin{figure}[htb]
        \centering
        \includegraphics{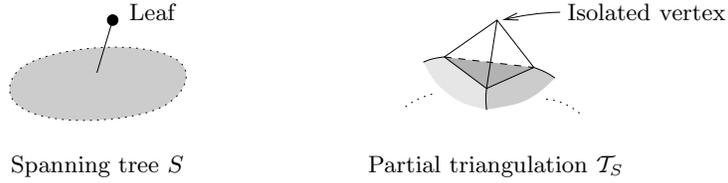}
        \caption{A tetrahedron of $\tri_S$ corresponding to a leaf in
            the spanning tree $S$}
        \label{fig-leaf}
    \end{figure}

    Every tree of size $n>2$ has at least two leaves; let $\ell$ be one
    such leaf, and let $\Delta_\ell$ be the corresponding tetrahedron in
    $\tri_S$.  Consider the three faces of $\Delta_\ell$ that surround
    the isolated vertex of $\ell$.  At least
    one of these faces must be joined to face of a \emph{different}
    tetrahedron in the final triangulation $\tri$; as a consequence,
    the isolated vertex of $\ell$ will be identified with some other
    tetrahedron vertex and we will have $v \leq v_S - 1 = n+2$
    vertices in total.

    We can repeat this argument upon a second leaf $\ell'$ to lower our
    bound once more, establishing the final result
    $v \leq v_S - 2 = n+1$.
    The only way this argument can fail is if both
    ``new'' vertex identifications are the same; that is, from our first
    leaf we find that the isolated vertex of $\ell$ is identified with the
    isolated vertex of $\ell'$, and then from our second leaf we find
    that the isolated vertex of $\ell'$ is identified with the isolated
    vertex of $\ell$.

    We are only forced into this redundancy if \emph{every} additional
    edge from $\ell$ in the complementary graph $S \backslash G$
    runs to $\ell'$ or is a loop back to $\ell$; likewise,
    every additional edge from $\ell'$ in $S \backslash G$ must run
    to $\ell$ or be a loop back to $\ell'$.
    In other words, we must have one of the two scenarios
    depicted in Figure~\ref{fig-redundant-leaves}.

    \begin{figure}[htb]
        \centering
        \includegraphics{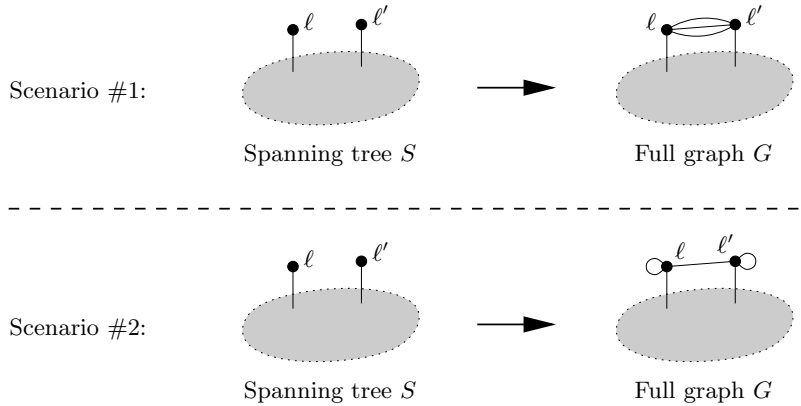}
        \caption{The two ``redundant'' scenarios in our analysis of leaves}
        \label{fig-redundant-leaves}
    \end{figure}

    Even still, we can avoid this redundancy if the tree $S$ has three or
    more leaves (we simply replace $\ell'$ with a different selection).
    In fact, given that we can choose \emph{any} spanning tree $S$, we
    are only forced into this redundancy if \emph{every} spanning tree
    within $G$ has precisely two leaves and gives one of the scenarios
    of Figure~\ref{fig-redundant-leaves}.  The only such connected
    4-valent multigraph $G$ on $n>2$ vertices is the graph depicted in
    Figure~\ref{fig-cycle-loops}; that is, a single $n$-cycle with a
    loop at every vertex.

    \begin{figure}[htb]
        \centering
        \includegraphics[scale=0.7]{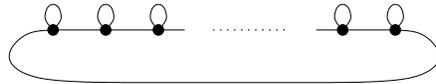}
        \caption{The only face pairing graph that forces redundancy in
            our leaf analysis}
        \label{fig-cycle-loops}
    \end{figure}

    For such a face pairing graph
    we can lower our bound from $n+3$ to $n+1$ as follows.
    Let $i$ be a non-leaf vertex of the tree $S$.  The full graph $G$
    has a loop at vertex $i$, which means that two distinct vertices of
    the corresponding tetrahedron in $\tri_S$ will be identified in the final
    triangulation $\tri$.  This identification does not involve
    the isolated vertices of the leaves, and so we can now return to our
    earlier argument on a single leaf to find a second (and different)
    identification between distinct vertices of $\tri_S$, showing that
    $v \leq v_S - 2 = n+1$.
\end{proof}

It can in fact be shown that this bound of $v \leq n+1$ is tight;
the proof involves a general construction for arbitrary $n$, and
we omit the details here.  For $n=2$ there is a closed 3-manifold
triangulation with $n+2 = 4$ vertices (this is the triangulation of the
3-sphere obtained by identifying the boundaries of two tetrahedra using the
identity mapping).


We proceed now to the main result of this section, which is a new
bound on the asymptotic growth rate of the size of the standard solution
set (that is, the number of vertices of the standard projective solution
space $\sproj$).

\begin{theorem} \label{t-std-bound}
    The size of the standard solution set is asymptotically bounded above by
    \[ O\left(\left[9 \cdot \left(
        \frac{1+\sqrt{13/9}}{2}\right)^5\right]^n\right)
        \quad \simeq \quad O(14.556^n). \]
\end{theorem}

\begin{proof}
    Let $\sigma$ denote the number of admissible vertices of the
    standard projective solution space $\sproj$.
    Following the analogous result in quadrilateral coordinates
    (Theorem~\ref{t-quad-bound}), our strategy is to bound $\sigma$
    by working through the maximal admissible faces of each dimension.
    As usual, we let $v$ denote the number of vertices in the underlying
    triangulation $\tri$.

    Once again we assume that $n \geq 3$ to avoid small-case
    anomalies.  Furthermore, we assume that the quadrilateral projective
    solution space $\qproj$ has at least one admissible vertex
    (otherwise it is simple to show that there are precisely $v \leq n+1$
    admissible vertices in $\sproj$, corresponding to the $v$ vertex links
    in $\tri$).

    Let $F$ be any maximal admissible face of $\sproj$.
    From Corollary~\ref{c-adm-facets} we know that $F$ has at most $5n$ facets.
    Furthermore, Lemma~\ref{l-quad-max-dim} and
    Lemma~\ref{l-bijection-std-quad} together show that
    $F$ has dimension $d+v$ for some $d$
    in the range $0 \leq d \leq n-1$.  Our immediate aim is to bound the
    number of vertices of $F$.  There are two cases to consider:
    \begin{itemize}
        \item If $d>0$ or $v>1$ then the dimension of $F$ is $\geq 2$,
        and we can combine McMullen's theorem with Lemma~\ref{l-ubt-k}
        to show that $F$ has at most $M_{d+v,5n}$ vertices.
        Using Lemma~\ref{l-min-vert} we then have
        $d+v \leq d+n+1 \leq 2n < 5n/2$, whereupon
        Lemma~\ref{l-ubt-d} gives us $M_{d+v,5n} \leq M_{d+n+1,5n}$.
        It follows that $F$ has at most $M_{d+n+1,5n}$ vertices.

        \item If $d=0$ and $v=1$ then $F$ is a 1-face (an edge)
        with precisely $2$ vertices.  It is simple to show that
        $2 \leq M_{n+1,5n} = M_{d+n+1,5n}$, so again $F$ has at most
        $M_{d+n+1,5n}$ vertices.
    \end{itemize}

    Once more we observe that each admissible vertex of $\sproj$
    is a vertex of some maximal admissible face, and so we can bound
    $\sigma$ by summing this bound of $M_{d+n+1,5n}$ over all maximal
    admissible faces.
    Lemma~\ref{l-quad-bound-dim} and Lemma~\ref{l-bijection-std-quad}
    together show that $\sproj$ has at most
    $3^{n-1-d}$ maximal admissible faces of dimension $d+v$ for each $d$,
    and so we have
    \begin{equation} \label{eqn-std-bound-sum}
        \sigma \leq \sum_{d=0}^{n-1} 3^{n-1-d} \cdot M_{d+n+1,5n}
        = \sum_{e=n+1}^{2n} 3^{2n-e} \cdot M_{e,5n}. 
    \end{equation}
    We can loosen this bound by extending the summation index
    $e$ to the full range $2 \leq e < 5n$, yielding
    \[
    \sigma \leq \sum_{e=2}^{5n-1} 3^{2n-e} \cdot M_{e,5n}
    = 9^n \sum_{e=2}^{5n-1} (1/3)^e \cdot M_{e,5n},
    \]
    whereupon Corollary~\ref{c-m-sums} gives us an asymptotic growth
    rate of
    \[ \sigma
        \in O\left(9^n \cdot \left[
            \frac{1+\sqrt{1+4/9}}{2} \right]^{5n} \right)
        = O\left(\left[9 \cdot \left(
            \frac{1+\sqrt{13/9}}{2}\right)^5\right]^n\right)
        \simeq O(14.556^n). \]
\end{proof}

We finish this section by applying our techniques to the important
case of a one-vertex triangulation.  In this case we are able to strip an
extra $3^n$ from our bound, yielding the following asymptotic result.

\begin{theorem} \label{t-std-bound-v1}
    If we restrict our attention to triangulations with precisely
    one vertex, then
    the size of the standard solution set is asymptotically bounded above by
    \[ O\left(\left[3 \cdot \left(
        \frac{1+\sqrt{13/9}}{2}\right)^5\right]^n\right)
        \quad \simeq \quad O(4.852^n). \]
\end{theorem}

\begin{proof}
    The argument is almost identical to the proof of
    Theorem~\ref{t-std-bound}, and we do not repeat the details here.
    The main difference arises in the derivation of
    equation~(\ref{eqn-std-bound-sum}):
    \begin{itemize}
        \item For the case $d>0$, we replace the bound $v \leq n+1$
        with the more precise $v=1$, allowing us to replace the term
        $M_{d+n+1,5n}$ with the tighter $M_{d+1,5n}$.

        \item For the case $d=0$, we cannot use McMullen's bound at
        all since we are looking at maximal admissible faces of
        dimension $d+v = 1$.
        Instead we note that every $1$-face is an edge with precisely
        two vertices, and we replace $M_{d+n+1,5n}$ with
        the constant $2$.
    \end{itemize}
    Separating out the cases $d>0$ and $d=0$,
    equation~(\ref{eqn-std-bound-sum}) then becomes
    \[
        \sigma \leq 2 \cdot 3^{n-1} +
            \sum_{d=1}^{n-1} 3^{n-1-d} \cdot M_{d+1,5n} \\
        = \frac23 \cdot 3^n + \sum_{e=2}^{n} 3^{n-e} \cdot M_{e,5n}.
    \]
    Again we extend the summation index $e$ to the full range
    $2 \leq e < 5n$, giving
    \[
    \sigma \leq \frac23 \cdot 3^n + \sum_{e=2}^{5n-1} 3^{n-e} \cdot M_{e,5n}
    = \frac23 \cdot 3^n + 3^n \sum_{e=2}^{5n-1} (1/3)^e \cdot M_{e,5n},
    \]
    whereupon Corollary~\ref{c-m-sums} shows the asymptotic growth rate to be
    \[ \sigma
        \in O\left(3^n + 3^n \cdot \left[
            \frac{1+\sqrt{1+4/9}}{2} \right]^{5n} \right)
        = O\left(\left[3 \cdot \left(
            \frac{1+\sqrt{13/9}}{2}\right)^5\right]^n\right)
        \simeq O(4.852^n). \]
\end{proof}

\section{Discussion} \label{s-discussion}

The complexity bounds of Sections~\ref{s-quad} and~\ref{s-std} are
significant improvements upon the prior state of the art.
The reason for this success is because we have been able to
integrate admissibility (in particular, the quadrilateral constraints)
with the high-powered machinery of polytope theory (in particular,
McMullen's upper bound theorem).  Previous results have either
used polytope theory on only a superficial level \cite{hass99-knotnp},
or else drawn on deeper polytope theory
but without any use of admissibility at all
\cite{burton10-complexity,burton10-dd}.

The difficulty in integrating admissibility with polytope theory arises
because the quadrilateral constraints are non-linear,
and the admissible region of each projective solution space is far from
being a convex polytope.  In this paper we circumvent these difficulties
by working with maximal admissible faces.  However, this leads to
certain inefficiencies, as we discuss further below.

It is known that any complexity bound on the size of the standard and
quadrilateral solution sets must be exponential, even if we restrict our
attention to one-vertex triangulations
\cite{burton10-complexity,burton10-extreme}.
However, the new bounds in this paper still leave significant room to move.
In standard coordinates the worst known cases grow with complexity
$O(17^{n/4}) \simeq O(2.03^n)$ in comparison to our $O(14.556^n)$;
see \cite{burton10-complexity} for details.\footnote{%
    These cases are constructed and analysed for all $n>5$, and
    experimental evidence supports the conjecture that these are
    the worst cases possible.}
In quadrilateral coordinates, comprehensive experimental evidence
from \cite{burton10-extreme} suggests that the worst cases
grow with complexity well
below $O(\phi^n) \simeq O(1.618^n)$, in contrast to our current
bound of $O(3.303^n)$.

This gap between theory and practice suggests that further research into
theoretical bounds could be fruitful.  The methods of this paper suggest
several potential avenues for improvement:
\begin{itemize}
    \item
    Because the proofs of Theorems~\ref{t-quad-bound} and~\ref{t-std-bound}
    iterate through each maximal admissible face, it is likely that we
    count each admissible vertex many times over.
    Finding a mechanism to avoid this multiple-counting could help tighten our
    bounds further.

    \item
    The key to all of the new bounds in this paper is
    Lemma~\ref{l-quad-bound-dim}, where we show that $\qproj$ has at
    most $3^{n-1-d}$ maximal admissible faces of each dimension $d \leq n-1$.
    This bound has been empirically tested against all $\sim 150$~million
    closed 3-manifold triangulations of size $n \leq 9$ (the same census
    used in \cite{burton10-complexity}), with intriguing results.

    \begin{table}[tb]
    \small
    \centering
    \begin{tabular}{c|r|r|r|r|r|r|r|r|r|r}
    Number of & \multicolumn{9}{c|}{Most maximal admissible faces
         of dimension \ldots } & \multicolumn{1}{c}{Number of} \\
    tetrahedra ($n$) & 0 & 1 & 2 & 3 & 4 & 5 & 6 & 7 & 8 & triangulations \\
    \hline
    1 & \textbf{1} & \phantom{000} & \phantom{000} & \phantom{000} &
        \phantom{000} & \phantom{000} & \phantom{000} & \phantom{000} & 
        \phantom{000} & 4 \\
    2 & \textbf{3} & \textbf{1} &&&&&&&& 17 \\
    3 & \phantom{00}4 & \textbf{3} & \textbf{1} &&&&&&& 81 \\
    4 & 5 & \textbf{9} & \textbf{3} & \textbf{1} &&&&&& 577 \\
    5 & 6 & 15 & \textbf{9} & \textbf{3} & \textbf{1} &&&&& 5\,184 \\
    6 & 4 & 22 & \textbf{27} & \textbf{9} & \textbf{3} & \textbf{1} &&&&
        57\,753 \\
    7 & 8 & 31 & 50 & \textbf{27} & \textbf{9} & \textbf{3} & \textbf{1} &&&
        722\,765 \\
    8 & 6 & 40 & 78 & \textbf{81} & \textbf{27} & \textbf{9} & \textbf{3} &
        \textbf{1} && 9\,787\,509 \\
    9 & 4 & 48 & 118 & 180 & \textbf{81} & \textbf{27} & \textbf{9} &
        \textbf{3} & \textbf{1} & 139\,103\,032
    \end{tabular}
    \caption{The largest number of maximal admissible faces of various
    dimensions}
    \label{tab-maxadm}
    \end{table}

    The outcomes of this testing are summarised in Table~\ref{tab-maxadm}.
    For high dimensions $d \geq \frac{n}{2}-1$,
    the bound of $\leq 3^{n-1-d}$ maximal admissible faces appears to be
    tight (these numbers appear in bold in the table).
    For low dimensions $d < \frac{n}{2}-1$ the number of maximal
    admissible faces drops away significantly, right down to what
    appears to be $O(n)$ maximal admissible faces of dimension $0$.

    As an exploratory exercise, for each $n \leq 9$ we can work through the
    original proof of Theorem~\ref{t-quad-bound} but replace each bound of
    $3^{n-1-d}$ maximal admissible $d$-faces with the corresponding figure
    from Table~\ref{tab-maxadm}.  The resulting bounds on the number of
    admissible vertices of $\qproj$ are shown in Table~\ref{tab-maxadm-expt},
    and their growth rate settles down to
    roughly $O(2.86^n)$, well below our current bound of
    $O(3.303^n)$.
    This suggests that, if we can tighten Lemma~\ref{l-quad-bound-dim}
    for low dimensions, we can significantly improve our bounds again.

    \begin{table}[tb]
    \small
    \centering
    \begin{tabular}{l|r|r|r|r|r|r|r|r|r}
    Number of tetrahedra ($n$) & 1 & 2 & 3 & 4 & 5 & 6 & 7 & 8 & 9 \\
    \hline
    Max.\ number of admissible vertices &
        1 & 5 & 13 & 39 & 104 & 315 & 859 & 2458 & 7018 \\
    \end{tabular}
    \caption{Empirical complexity bounds based on
        the results of Table~\ref{tab-maxadm}}
    \label{tab-maxadm-expt}
    \end{table}

    \item
    Finally, even for high-dimensional faces where Lemma~\ref{l-quad-bound-dim}
    does appear to be tight, we know from Lemma~\ref{l-quad-simultaneous} that
    equality cannot hold for all high dimensions \emph{simultaneously}.
    Empirical testing again suggests that Lemma~\ref{l-quad-simultaneous} is
    merely one example of a larger set of constraints, and exploring
    these constraints may yield more useful information about the
    structure and number of maximal admissible faces.
\end{itemize}

For a final observation, we return to the worst known cases in standard
coordinates.
These are pathological triangulations of the 3-sphere for arbitrary $n > 5$,
each with $O(17^{n/4})$ admissible vertices in $\sproj$, and there is strong
empirical evidence \cite{burton10-complexity}
to suggest that this family of triangulations
yields the largest number of vertices for all $n$.

What is interesting about these cases is each triangulation has
\emph{only one} maximal admissible face.  In quadrilateral coordinates
this maximal face is just an $(n-1)$-simplex, and the quadrilateral
projective solution space $\qproj$ has only $n$ admissible vertices in
total.  In other words, for these cases
the pathological complexity only appears in the
extension to standard coordinates.  These observations suggest that a
better understanding of the relationships between the face lattices in
$\sproj$ and $\qproj$ could be an important step in achieving stronger
bounds on the complexities of these polytopes.

%
%

\section*{Acknowledgements}

The author is grateful to the Australian Research Council for their support
under the Discovery Projects funding scheme (project DP1094516).

%
%

\small
\bibliographystyle{amsplain}
\bibliography{pure}

\newcommand{\noopsort}[1]{}
\providecommand{\bysame}{\leavevmode\hbox to3em{\hrulefill}\thinspace}
\providecommand{\MR}{\relax\ifhmode\unskip\space\fi MR }
\providecommand{\MRhref}[2]{%
  \href{http://www.ams.org/mathscinet-getitem?mr=#1}{#2}
}
\providecommand{\href}[2]{#2}
\begin{thebibliography}{10}

\bibitem{agol02-knotgenus}
Ian Agol, Joel Hass, and William Thurston, \emph{3-manifold knot genus is
  {NP}-complete}, STOC '02: Proceedings of the Thiry-Fourth Annual {ACM}
  Symposium on Theory of Computing, ACM Press, 2002, pp.~761--766.

\bibitem{burton04-facegraphs}
Benjamin~A. Burton, \emph{Face pairing graphs and 3-manifold enumeration}, J.
  Knot Theory Ramifications \textbf{13} (2004), no.~8, 1057--1101.

\bibitem{burton09-convert}
\bysame, \emph{Converting between quadrilateral and standard solution sets in
  normal surface theory}, Algebr. Geom. Topol. \textbf{9} (2009), no.~4,
  2121--2174.

\bibitem{burton10-quadoct}
\bysame, \emph{Quadrilateral-octagon coordinates for almost normal surfaces},
  To appear in Experiment. Math., \texttt{arXiv:\allowbreak 0904.3041}, April
  2009.

\bibitem{burton10-complexity}
\bysame, \emph{The complexity of the normal surface solution space}, SCG '10:
  Proceedings of the Twenty-Sixth Annual Symposium on Computational Geometry,
  ACM, 2010, pp.~201--209.

\bibitem{burton10-extreme}
\bysame, \emph{Extreme cases in normal surface enumeration}, In preparation,
  2010.

\bibitem{burton10-dd}
\bysame, \emph{Optimizing the double description method for normal surface
  enumeration}, Math. Comp. \textbf{79} (2010), no.~269, 453--484.

\bibitem{burton09-ws}
Benjamin~A. Burton, J.~Hyam Rubinstein, and Stephan Tillmann, \emph{The
  {W}eber-{S}eifert dodecahedral space is non-{H}aken}, To appear in Trans.
  Amer. Math. Soc., \texttt{arXiv:\allowbreak 0909.4625}, September 2009.

\bibitem{callahan99-cuspedcensus}
Patrick~J. Callahan, Martin~V. Hildebrand, and Jeffrey~R. Weeks, \emph{A census
  of cusped hyperbolic 3-manifolds}, Math. Comp. \textbf{68} (1999), no.~225,
  321--332.

\bibitem{dyer83-complexity}
M.~E. Dyer, \emph{The complexity of vertex enumeration methods}, Math. Oper.
  Res. \textbf{8} (1983), no.~3, 381--402.

\bibitem{grunbaum03}
Branko Gr{\"u}nbaum, \emph{Convex polytopes}, 2nd ed., Graduate Texts in
  Mathematics, no. 221, Springer, New York, 2003.

\bibitem{haken61-knot}
Wolfgang Haken, \emph{Theorie der {N}ormalfl{\"a}chen}, Acta Math. \textbf{105}
  (1961), 245--375.

\bibitem{haken62-homeomorphism}
\bysame, \emph{{\"U}ber das {H}om{\"o}omorphieproblem der
  3-{M}annigfaltigkeiten. {I}}, Math. Z. \textbf{80} (1962), 89--120.

\bibitem{hass99-knotnp}
Joel Hass, Jeffrey~C. Lagarias, and Nicholas Pippenger, \emph{The computational
  complexity of knot and link problems}, J. Assoc. Comput. Mach. \textbf{46}
  (1999), no.~2, 185--211.

\bibitem{hempel76}
John Hempel, \emph{3-manifolds}, Annals of Mathematics Studies, no.~86,
  Princeton University Press, Princeton, NJ, 1976.

\bibitem{jaco84-haken}
William Jaco and Ulrich Oertel, \emph{An algorithm to decide if a
  {$3$}-manifold is a {H}aken manifold}, Topology \textbf{23} (1984), no.~2,
  195--209.

\bibitem{jaco03-0-efficiency}
William Jaco and J.~Hyam Rubinstein, \emph{0-efficient triangulations of
  3-manifolds}, J. Differential Geom. \textbf{65} (2003), no.~1, 61--168.

\bibitem{jaco95-algorithms-decomposition}
William Jaco and Jeffrey~L. Tollefson, \emph{Algorithms for the complete
  decomposition of a closed {$3$}-manifold}, Illinois J. Math. \textbf{39}
  (1995), no.~3, 358--406.

\bibitem{kang04-taut1}
Ensil Kang and J.~Hyam Rubinstein, \emph{Ideal triangulations of 3-manifolds
  {I}; {S}pun normal surface theory}, Proceedings of the Casson Fest, Geom.
  Topol. Monogr., vol.~7, Geom. Topol. Publ., Coventry, 2004, pp.~235--265.

\bibitem{khachiyan08-hard}
Leonid Khachiyan, Endre Boros, Konrad Borys, Khaled Elbassioni, and Vladimir
  Gurvich, \emph{Generating all vertices of a polyhedron is hard}, Discrete
  Comput. Geom. \textbf{39} (2008), no.~1-3, 174--190.

\bibitem{kneser29-normal}
Hellmuth Kneser, \emph{Geschlossene {F}l{\"a}chen in dreidimensionalen
  {M}annigfaltigkeiten}, Jahresbericht der Deut. Math. Verein. \textbf{38}
  (1929), 248--260.

\bibitem{martelli02-decomp}
Bruno Martelli and Carlo Petronio, \emph{A new decomposition theorem for
  3-manifolds}, Illinois J. Math. \textbf{46} (2002), 755--780.

\bibitem{matveev90-complexity}
Sergei~V. Matveev, \emph{Complexity theory of three-dimensional manifolds},
  Acta Appl. Math. \textbf{19} (1990), no.~2, 101--130.

\bibitem{mcmullen70-ubt}
P.~McMullen, \emph{The maximum numbers of faces of a convex polytope},
  Mathematika \textbf{17} (1970), 179--184.

\bibitem{rubinstein95-3sphere}
J.~Hyam Rubinstein, \emph{An algorithm to recognize the {$3$}-sphere},
  Proceedings of the International Congress of Mathematicians ({Z}{\"u}rich,
  1994), vol.~1, Birkh{\"a}user, 1995, pp.~601--611.

\bibitem{stanley97-vol1}
Richard~P. Stanley, \emph{Enumerative combinatorics, {V}ol. 1}, Cambridge
  Studies in Advanced Mathematics, no.~49, Cambridge University Press,
  Cambridge, 1997.

\bibitem{thompson94-thinposition}
Abigail Thompson, \emph{Thin position and the recognition problem for {$S\sp
  3$}}, Math. Res. Lett. \textbf{1} (1994), no.~5, 613--630.

\bibitem{tillmann08-finite}
Stephan Tillmann, \emph{Normal surfaces in topologically finite 3-manifolds},
  Enseign. Math. (2) \textbf{54} (2008), 329--380.

\bibitem{tollefson98-quadspace}
Jeffrey~L. Tollefson, \emph{Normal surface {$Q$}-theory}, Pacific J. Math.
  \textbf{183} (1998), no.~2, 359--374.

\bibitem{ziegler95}
G{\"u}nter~M. Ziegler, \emph{Lectures on polytopes}, Graduate Texts in
  Mathematics, no. 152, Springer-Verlag, New York, 1995.

\end{thebibliography}

%
%

\bigskip
\noindent
Benjamin A.~Burton \\
School of Mathematics and Physics, The University of Queensland \\
Brisbane QLD 4072, Australia \\
(bab@maths.uq.edu.au)

\end{document}